\documentclass[11pt]{article}

\usepackage{amsmath,amssymb,amsthm}
\usepackage{color}
\usepackage{graphicx}
\newtheorem{theorem}{Theorem}

\newtheorem{definition}{Definition}
\newtheorem{lemma}{Lemma}
\newtheorem{cor}{Corollary}
\newtheorem{proposition}{Proposition}
\newtheorem{remark}{Remark}

\def \beq{ \begin{equation}}
\def \eeq{\end{equation}}
\renewcommand{\labelitemi}{--}
\def \ind{\mbox{ind}}

\title{A new method to study relative equilibria on $\mathbb{S}^2$}
\date{}  
\begin{document}
	\maketitle
	\author{\begin{center}
	{ Toshiaki~Fujiwara$^1$, Ernesto P\'{e}rez-Chavela$^2$}\\	
		\bigskip
	   $^1$College of Liberal Arts and Sciences, Kitasato University,       Japan. fujiwara@kitasato-u.ac.jp\\
	    $^2$Department of Mathematics, ITAM, M\'exico.\\ ernesto.perez@itam.mx
	\end{center}
	
\date{}

\bigskip

\begin{abstract}
We develop a new geometrical technique to study relative equilibria for a system of $n$--positive masses, moving on the two dimensional sphere $\mathbb{S}^2$, under the influence of a general potential which only depends on the mutual distances among the masses. 
The big difficulty to study relative equilibria on 
$\mathbb{S}^2$, that we call $RE$ by short, is the absence of the center of mass as a first integral.
We show that the two vanishing components of the angular momentum, for motions on $\mathbb{S}^2$, play the same role as
the center of mass for motions on the Euclidean plane. 
From here we obtain that the rotation axis of a $RE$ 
is one of the principal axes of the inertia tensor. 
Conditions for have $RE$ 
and relations between the shape 
(given by the arc angles $\sigma_{ij}$ among the masses) and the configuration (given by the polar angles $\theta_k$ and $\phi_i - \phi_j$ in spherical coordinates) are shown. 
For $n=3$, we show explicitly the conditions to have Euler and Lagrange $RE$ on $\mathbb{S}^2$. 
As an application of our method we study the the equal masses case for the positive curved three body problem
where we show the existence of scalene and isosceles Euler
$RE$ and isosceles Lagrange $RE$.
\end{abstract}


{\bf Math. Subject Class 2020:} 70F07, 70F10, 70F15


\section{Introduction}
A relative equilibrium 
($RE$)
on the Euclidean plane, is a solution of a system of $n$--positive point masses, where each mass is rotating uniformly around the center of mass of the system  with the same angular velocity.  
The mutual distances among the masses remain constant along the motion. The masses behave as if they belong to a rigid body. For the classical 
planar
Newtonian $3$--body problem, it is well known that there are five classes of relative equilibria, three collinear or Euler $RE$ and two equilateral triangle or Lagrange $RE$ \cite{Euler, Moeckel}. 
The configuration
of the masses in a $RE$ is called a {\it central configuration}, we obtain the respective relative equilibrium by taking a particular uniform rotation through the center of mass \cite{Wintner}.

When we extend the concept of relative equilibria to the sphere 
$\mathbb{S}^2$, the main problem for its analysis is the absence of the center of mass as a first integral (see for instance \cite{Borisov1, Borisov2, Diacu-EPC1, Diacu1, Diacu3, EPC1, Shchepetilov2} and the references therein). 
In his monograph  
on relative equilibria for the curved $n$--body problem \cite{Diacu1}, F.~Diacu wrote {\it Their absence, however, complicates the study of the problem since many of the standard methods used in the classical case don't apply any more.}
At that time, no method was known
to determine the axis of rotation  
for given masses $m_k$ and 
a shape (the set of arc angles $\sigma_{ij}$
between the bodies $i$ and $j$).

The goal of this article is to develop a new geometrical method to study $RE$ on the sphere. To achieve this goal, 
two questions we asked ourselves were very important:
\begin{itemize}
\item[i)] Is there any first integral that can determine the candidate for the rotation axis for a  shape? 
\item[ii)] What is the relation between a shape and the rotation axis?
\end{itemize}

The answer to the first question is yes, the first integral is given by two components of the angular momentum. 
The answer to the second question is that,
the rotation axis is one of the principal axes (the eigenvectors)
of the inertia tensor \cite{Routh}.
It is well known in the rigid body problem
\cite{Goldstein, Hestenes, LandauLifshitz, Routh},
that if the rotation axis is constant in time then the rotation axis is 
one of the principal axes of the inertia tensor.
Such a rotation axis is called ``a permanent axis of rotation'' \cite{Routh}.
Since for $RE$ the bodies behave as a rigid body, the inertia tensor is 
important in its 
analysis, as we will see ahead in the manuscript.
This is the key point to understand $RE$ on the sphere. 
As far as we know, this is the first time that the analysis of $RE$ is formulated in this context. 
Our method gives a systematic way to find $RE$ on $\mathbb{S}^2$.

After the introduction, the paper is organized as follows: In Section \ref{prelims}
we get the equations of motion for $n$ particles with positive masses moving on $\mathbb{S}^2$ in spherical coordinates, we also obtain the equations for the angular momentum, which will play a main role in our geometric analysis of the $RE$.

In Section \ref{sec3} we introduce the geometry of relative equilibria on $\mathbb{S}^2$.
As in the rigid body problem, we show that if the angular velocity $\omega \neq 0$, then the rotation axis is one of the principal axis of the inertia tensor.
We show that two components of the angular momentum can be seen as an extension of the center of mass to the sphere.

A collinear $RE$ is a relative equilibrium where the $n$--bodies are on the same geodesic, if this is not the case we call them, non-collinear $RE$, the only restriction on the masses is that they must be positive. Our results for the collinear case can be generalized to the general $n$--body problem, but for the non-collinear case the analysis holds just for the case $n=3$. 
To avoid complications in the notation, we will restrict our study
for both configurations, collinear and non collinear to the case $n=3$.
In reference to the Euclidean case we will call them Eulerian relative equilibria for collinear case ($ERE$ by short), and Lagrangian relative equilibria for the   non collinear case ($LRE$ by short).  

The analysis of relative equilibria is naturally separated into two groups $ERE$ and $LRE$.
In Section \ref{sec:Euler} we deduce the equations of motion for the $ERE$. We give the necessary and sufficient conditions for a shape
to generate an $ERE$, see Theorem \ref{propConditionForShape}. The same analysis for the $LRE$ is done in Section \ref{sec:Lagrange}, where we also 
give the necessary and sufficient conditions for a shape, to generate a $LRE$, see Theorem \ref{ConditonForLRE}.

In order to have concrete examples to show how our method works, 
in Section \ref{sec:curved}, 
we study the
 three-body problem on $\mathbb{S}^2$ with equal masses moving
under the influence of the cotangent potential
(see for instance \cite{Borisov2, Diacu-EPC1, Diacu1}).
We show some new families of $ERE$ and $LRE$ in this problem. 
Finally in Section \ref{conclusions}, we summarize our results and state some final remarks.
 

\section{Preliminaries and equations of motion}\label{prelims}
We use spherical coordinates to describe the $n$--body problem on $\mathbb{S}^2$. First, we introduce the notations that we will use along the paper and compute the angular momentum.

\subsection{Notations}\label{notations}
The point $(X,Y,Z)$ on $\mathbb{S}^2$ with radius $R$ is represented by the spherical coordinates $(R,\theta,\phi)$,
that is, 
$(X,Y,Z)=R(\sin\theta\cos\phi,\sin\theta\sin\phi,\cos\theta)$.
The chord length $D_{ij}$ between the  points
$(X_i,Y_i,Z_i)$ and $(X_j,Y_j,Z_j)$ is given by
$D_{ij}^2
=(X_i-X_j)^2+(Y_i-Y_j)^2+(Z_i-Z_j)^2
=2R^2\big(1-\cos\theta_i\cos\theta_j-\sin\theta_i\sin\theta_j\cos(\phi_i-\phi_j)
		\big)$.
The arc angle $\sigma_{ij}$ ($0\le \sigma_{ij}\le \pi$)
is equal to the angle between the two points 
as seen from the center of $\mathbb{S}^2$,
which is related to $D_{ij}$ by
$\sin(\sigma_{ij}/2)=D_{ij}/(2R)$.

From the above, we obtain the relation
\begin{equation}
\label{fundamentalrelation}
\cos\sigma_{ij}
=\cos\theta_i\cos\theta_j+\sin\theta_i\sin\theta_j\cos(\phi_i-\phi_j).
\end{equation}

\subsection{Equations of motion}

The Lagrangian for the $n$--body problem on
$\mathbb{S}^2$
is given by
\begin{equation}\label{theLagrangian}
L=K+V, \end{equation}
$$ \text{where} \quad 
K=R^2\sum\nolimits_k\frac{m_k}{2}
		\left(\dot{\theta}_k^2+\sin^2(\theta_k)\dot{\phi}_k^2\right),
\quad
V=\sum\nolimits_{i<j}\frac{m_i m_j}{R} U(\cos\sigma_{ij}),
$$
here dot on symbols represents time derivative. 

For the derivative of the potential $U$,
we use the notation
$U'(\cos\sigma_{ij})=dU(\cos\sigma_{ij})/d(\cos\sigma_{ij})$.
We  assume that $U'(\cos\sigma_{ij})$ is continuous and has definite sign for all ranges of $\sigma_{ij} \in (0,\pi)$. 
The sign $U'(\cos \sigma_{ij})>0$ stands for attractive force, and $<0$ for repulsive force.

The equations of motion are derived from the above Lagrangian through the Euler-Lagrange equations.

Since the Lagrangian is  invariant under $SO(3)$ rotation
around the centre of $\mathbb{S}^2$,
the angular momentum 
${\bf c} 
= (c_x, c_y, c_z)
=\sum_k m_k (X_k,Y_k,Z_k)\times (\dot{X}_k, \dot{Y}_k, \dot{Z}_k)$
is a first integral \cite{Diacu-EPC1, Diacu1}.
Each component is represented as
\begin{align}
c_x&=R^2 \sum\nolimits_k m_k \left(-\sin(\phi_k) \dot\theta_k
	-\sin(\theta_k)\cos(\theta_k)\cos(\phi_k) \dot\phi_k\right),
	\label{defCx}\\
c_y&=R^2 \sum\nolimits_k m_k \left(\cos(\phi_k) \dot\theta_k
	-\sin(\theta_k)\cos(\theta_k)\sin(\phi_k) \dot\phi_k\right),
	\label{defCy}\\
c_z&=R^2 \sum\nolimits_k m_k \sin^2(\theta_k) \dot \phi_k.
	\label{defCz}
\end{align}


\section{Geometry of relative equilibria}\label{sec3}
A relative equilibrium is a solution of the equations of motion where each mass is rotating uniformly with the same angular velocity, the motion is like a rigid body. We give a formal definition of relative equilibria on $\mathbb{S}^2$ in terms of the spherical coordinates.

\begin{definition}[Relative equilibrium]\label{def1}
A relative equilibrium on $\mathbb{S}^2$ is a solution 
of the equations of motion which 
satisfies
$\dot{\theta}_k=0$ and $\dot{\phi}_i-\dot{\phi}_j=0$ or 
$\phi_k(t) = \phi_k(0)+ \omega t$
for all $k=1,2,\cdots, n$ 
and all pair $(i,j),$ with $\omega$ constant, if the $z$--axis is properly chosen in a spherical coordinate system.
\end{definition}

The above definition is coordinate independent. The existence of such 
$z$--axis is the condition for the relative equilibrium. 
Unlike to the Euclidean plane, the $RE$ on $\mathbb{S}^2$ contain the following special solution.

\begin{definition}[Fixed point]\label{fixed-point}
A fixed point is a $RE$
with $\omega=0$.
\end{definition}

\begin{definition}[Collinear and non-collinear]\label{coll-and non}
A collinear $RE$ 
is a $RE$ where all $n$--bodies are on the same geodesic. If this is not the case, we call it non-collinear $RE$.
For $n=3$, a collinear $RE$ is called Eulerian $RE$ ($ERE$), the non collinear $RE$ are called Lagrangian $RE$ ($LRE$).
\end{definition}

\begin{figure}   \centering
   \includegraphics[width=4cm]{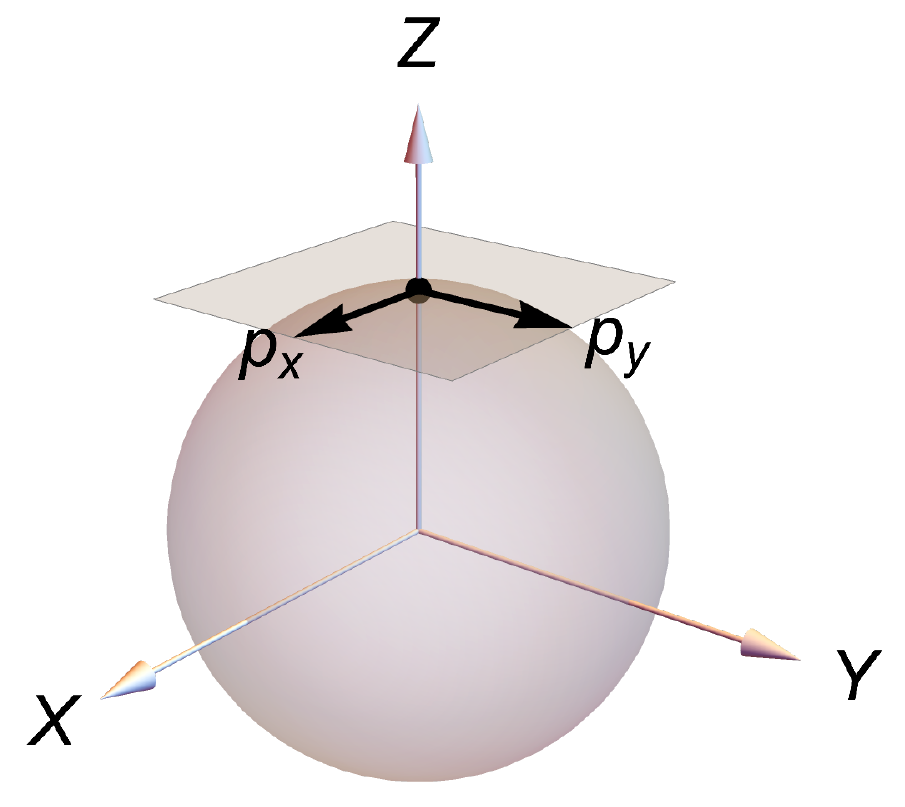} 
   \caption{The sphere
   and the tangent plane  at the north pole.}
   \label{figcxcyANDpxpy}
\end{figure}
Before to start our geometrical analysis of the $RE$ on the sphere, we must remember some important facts about the center of mass for the $n$--body problem in the Euclidean case, whose existence is based on  the existence of the integral of the linear momentum, denoted by
$(p_x,p_y)$. Which in turn is based  
 on the translational invariance of the Euclidean plane
$(x_k, y_k)\to (x_k+\delta x, y_k+\delta y)$.
This invariance comes from 
the fact that
``there are no special points on the Euclidean plane''.
It also has the angular momentum $C_z$, since 
``the plane has no special direction''.

The corresponding fact for the sphere
is that ``there are no any special points or directions on the surface.
So, $\mathbb{S}^2$ has $SO(3)$ rotational invariance,
that produces the angular momentum
$(c_x,c_y,c_z)$ integral.

To see the relation between the above two integrals,
consider a sphere  
and the tangent plane at the north pole,
(See Figure~\ref{figcxcyANDpxpy}).

Consider a region inside the arc length $R\theta$ 
from the north pole on $\mathbb{S}^2$.
When the radius of the sphere goes to infinity ($R\to \infty$)
keeping $r=R\theta$,
the region will coincides with the circle with radius $r$ on the tangent plane.
Thus we reach the following result.
\begin{proposition}[Relation between the  integrals]
For the limit $R\to \infty$ keeping $r_k=R\theta_k$ 
the integrals on $\mathbb{S}^2$ goes to 
the integrals on the Euclidean plane,
namely $(c_x,c_y)/R \to (-p_y,p_x)$
and $c_z \to C_z$.
\end{proposition}
\begin{proof}
A direct calculation shows that
$(c_x,c_y)/R
\to
\sum_k m_k\Big(
-(\dot{r}_k\sin\phi_k+r_k\cos(\phi_k)\dot\phi_k),
(\dot{r}_k\cos\phi_k-r_k\sin(\phi_k)\dot\phi_k)\Big)
=(-p_y,p_x)$
and
$c_z \to \sum_k m_k r_k^2 \dot\phi_k= C_z$.
\end{proof}
For $RE$, the following corollary gives
the direct relation between $c_x,c_y$ and the center of mass.
\begin{cor}\label{corCxCy}
For a $RE$,
the angular momenta $c_x$ and $c_y$ can be seen as an extension of the center of mass on the plane to the sphere.
\end{cor}
\begin{proof}
By  Definition \ref{def1}, for a $RE$,
the angular momentum ${\bf c}$, after the substitution 
$\dot{\theta}_k=0$ and $\dot{\phi}_k(t) = \omega$
has the form 
${\bf c} = (c_x, c_y, c_z)=(0,0,c_z),$ where
\begin{equation}\label{comp-am}
\begin{split}
(c_x,c_y)&= -R^2 \omega 
\sum\nolimits_k m_k \sin(\theta_k)\cos(\theta_k) \left(\cos(\phi_k),\sin(\phi_k)\right),\\
c_z&=R^2 \omega 
\sum\nolimits_k m_k \sin^2(\theta_k). 
\end{split}
\end{equation}
Therefore, 
for $R \to \infty$ keeping $r_k=R\theta_k$ 
the condition  
$(c_x,c_y)/R=0$
goes to
$- \omega \sum_k m_k r_k (\cos(\phi_k),\sin(\phi_k))=0$
which is the center of mass condition on the Euclidean plane.
\end{proof}

From here on, for simplicity we take $R=1$. 
The condition 
 $(c_x,c_y)=\nolinebreak 0$
plays an
important role for  the analysis of $RE$ on $\mathbb{S}^2$,
but even more important is the inertia tensor defined as
\begin{definition}
The inertia tensor is defined by the matrix
\begin{equation}\label{defI}
I=\left(\begin{array}{ccc}
I_{xx} & I_{xy} & I_{xz} \\
I_{yx} & I_{yy} & I_{yz} \\
I_{zx} & I_{zy} & I_{zz}
\end{array}\right).
\end{equation}
Where 
\begin{equation}
\label{inertiaTensorDiagonal}
\begin{split}
I_{xx}&=\sum\nolimits_{k} m_k (Y_k^2+Z_k^2)
	=\sum\nolimits_k m_k \Big(\cos^2(\theta_k)+\sin^2(\theta_k)\sin^2(\phi_k)\Big),\\
I_{yy}&=\sum\nolimits_{k} m_k (Z_k^2+X_k^2)
	=\sum\nolimits_k m_k\Big(\cos^2(\theta_k)+\sin^2(\theta_k)\cos^2(\phi_k)\Big),\\
I_{zz}&=\sum\nolimits_{k} m_k (X_k^2+Y_k^2)
	=\sum\nolimits_k m_k \sin^2(\theta_k),\\
\end{split}
\end{equation}
and
\begin{equation}
\label{inertiaTensorOffDiagonal}
\begin{split}
I_{xy}=I_{yx}
	&= -\sum\nolimits_{k} m_k X_k Y_k
	=-\sum m_k \sin^2(\theta_k)\sin\phi_k\cos\phi_k,\\
I_{xz}=I_{zx}
	&= -\sum\nolimits_{k} m_k X_k Z_k
	=-\sum m_k \sin\theta_k\cos\theta_k\cos\phi_k,\\
I_{yz}=I_{zy}
	&= -\sum\nolimits_{k} m_k Y_k Z_k
	=-\sum m_k \sin\theta_k\cos\theta_k\sin\phi_k.\\
\end{split}
\end{equation}
\end{definition}

Since $I$ is a symmetric real matrix, it has three real eigenvalues and three mutually orthogonal eigenvectors called principal axes of $I$.

The inertia tensor $I$ is transformed as a second-order tensor under the rotation of the coordinate axes.
Among the rotations, 
the following expression of $I$ is useful for understanding
the three body $RE$ on the sphere $\mathbb{S}^2$,
\begin{equation}\label{defJ}
J=\left(\begin{array}{ccc}
m_2+m_3 & -\sqrt{m_1m_2}\cos\sigma_{12} & -\sqrt{m_1m_3}\cos\sigma_{13} \\
-\sqrt{m_2m_1}\cos\sigma_{21} &m_3+m_1& -\sqrt{m_2m_3}\cos\sigma_{23} \\
-\sqrt{m_3m_1}\cos\sigma_{31} & -\sqrt{m_3m_2}\cos\sigma_{32} & m_1+m_2
\end{array}\right).
\end{equation}
It was not so easy to obtain the matrix J, to get it, we have used some trigonometric identities and a lot of computations, that we avoid in this manuscript.

\begin{lemma}\label{matrixJ}
For the three body problem on $\mathbb{S}^2$ the matrix $J$ defined by 
\eqref{defJ}
is similar to the inertia tensor $I$.
\end{lemma}

\begin{proof}
To show the similarity, we first calculate the characteristic polynomial $P(\lambda)$ of the inertia tensor $I$. Since this polynomial is rotation invariant, we can choose any spherical coordinate system to get this computation.  
We take
$(\theta_3,\phi_3)=(0,0)$,
$(\theta_1,\phi_1)=(\sigma_{31},0)$,
and $(\theta_2,\phi_2)=(\sigma_{23},\alpha)$ (See Figure \ref{figPutTheTriangle})
where
\begin{equation*}
\cos\alpha
=\frac{\cos\sigma_{12}-\cos\sigma_{31}\cos\sigma_{23}}
	{\sin\sigma_{31}\sin\sigma_{23}}.
\end{equation*}

\begin{figure}
   \centering
   \includegraphics[width=4cm]{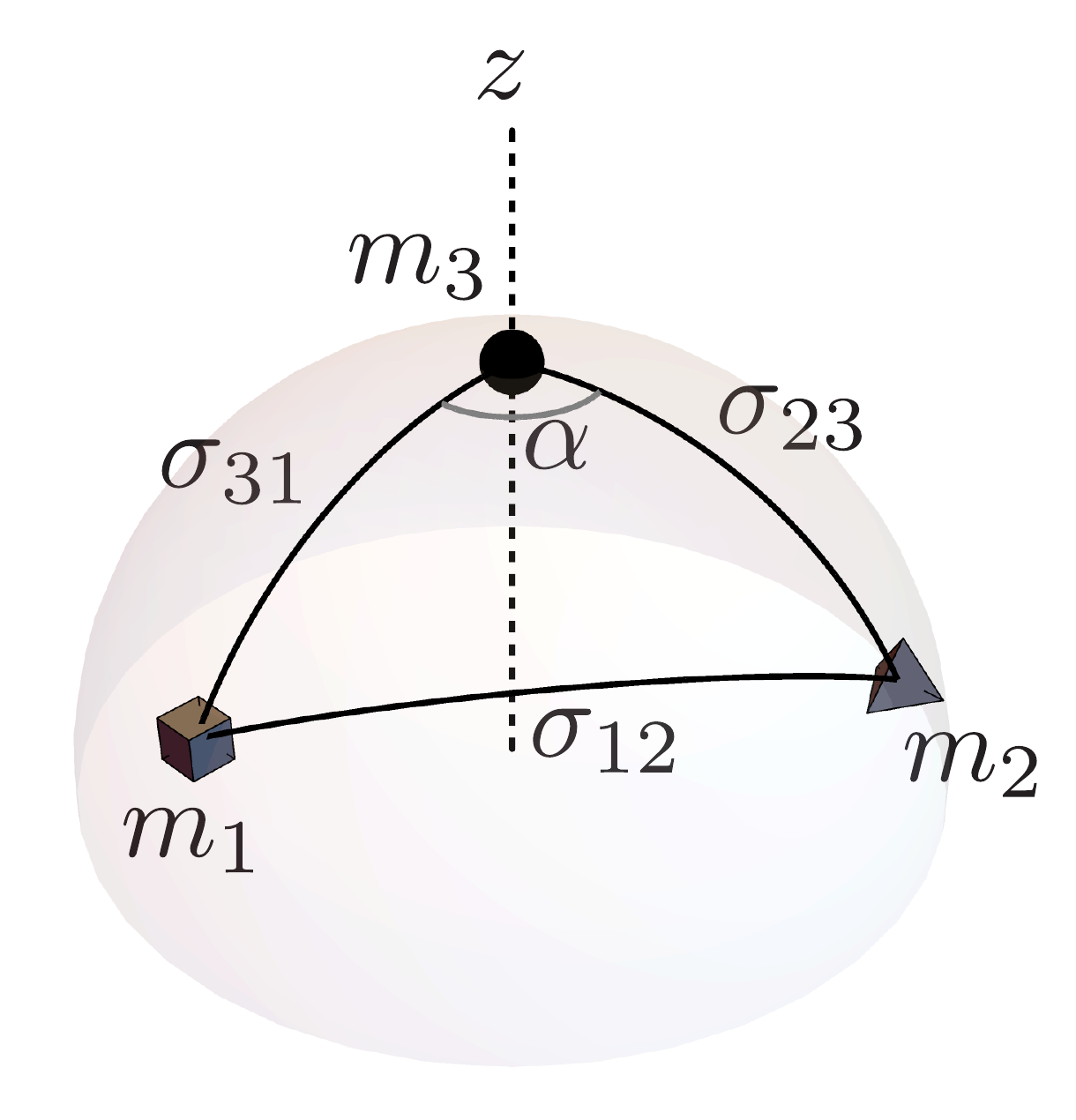} 
   \caption{The triangle with
   $\sigma_{ij}$ is placed to the temporal position
   $(\theta_3,\phi_3)=(0,0)$,
   $(\theta_1,\phi_1)=(\sigma_{31},0)$
   and
    $(\theta_2,\phi_2)=(\sigma_{23},\alpha)$.
    The cube, tetrahedron, and ball represent
    $m_1$, $m_2$, and $m_3$ respectively. 
    We will continue using this convention in the next figures.
    }
   \label{figPutTheTriangle}
\end{figure}

Using the above coordinate system, a direct computation yields
\begin{equation*}
\begin{split}
P(\lambda)=&
\Big(\lambda-(m_1+m_2)\Big)
	\Big(\lambda-(m_2+m_3)\Big)
	\Big(\lambda-(m_3+m_1)\Big)\\
	&-\Big(\lambda-(m_1+m_2)\Big)m_1m_2\cos^2(\sigma_{12})
	-\Big(\lambda-(m_2+m_3)\Big)m_2m_3\cos^2(\sigma_{23})\\
	&-\Big(\lambda-(m_3+m_1)\Big)m_3m_1\cos^2(\sigma_{31})
+2m_1m_2m_3\cos\sigma_{12}\cos\sigma_{23}\cos\sigma_{31}.
\end{split}
\end{equation*}

We can verify easily that the characteristic polynomial for $J$, has exactly the same expression as above. Since both $I$ and $J$ are real symmetric matrices and they have the same characteristic polynomial, they must be similar matrices. 
This finish the proof of Lemma \ref{matrixJ}.
\end{proof}

The difference between $I$ and $J$
is the coordinate system for
the same shape $\sigma_{ij}$.
In other words, the matrix $J$ is 
an expression for the inertia tensor $I$. 

In the following,
we will use $I$ in the equations (\ref{defI}, \ref{inertiaTensorDiagonal}, \ref{inertiaTensorOffDiagonal}) for the generic expression of the inertia tensor
that depends on the choice of the coordinate system,
and $J$ in equation \eqref{defJ} for the fixed expression.

A direct calculation yields the following identity:
For the column vector
$v=(\sqrt{m_1}\cos\theta_1,\sqrt{m_2}\cos\theta_2,
\sqrt{m_3}\cos\theta_3)^T$,
\begin{equation}\label{identityI}
v^TJv=
\left(\sum_{\ell=1,2,3}m_\ell\cos^2(\theta_\ell)\right)
\left(\sum_{\ell=1,2,3}m_\ell\sin^2(\theta_\ell)\right)
-(I_{xz}^2+I_{yz}^2),
\end{equation}
where ($^T$)
 represents the transpose.

\begin{lemma}\label{threeEquivalentStatements}
The following three statements are equivalent.
\begin{itemize}
\item[S1:] The $z$--axis is one of the principal axis of $I$.
\item[S2:] $I_{xz}=I_{yz}=0$.
\item[S3:] $\Psi_\theta
=(\sqrt{m_1}\cos\theta_1,\sqrt{m_2}\cos\theta_2,\sqrt{m_3}\cos\theta_3)^T/
\sqrt{\sum_{\ell=1,2,3}m_\ell \cos^2(\theta_\ell)}$
is the eigenvector of $J$ 
that 
belongs to the eigenvalue
$\lambda=\sum_{\ell=1,2,3}m_\ell \sin^2(\theta_\ell)$.
\end{itemize}
\end{lemma}

\begin{proof}[Proof of the equivalence of S1 and S2:]
Let $e_z=(0,0,1)$. Then 
$I e_z=(I_{xz},I_{yz},I_{zz})$.
Therefore
$I e_z = \lambda e_z$ 
is equivalent to  $I_{xz}=I_{yz}=0$.
\end{proof}

\begin{proof}[Proof of the equivalence of S2 and S3:]
The condition $I_{xz}=I_{yz}=0$ has the following alternative expression,
\begin{equation}\label{alternative}
\sum_{i=1,2,3} 
m_i \sin(\theta_i)\cos(\theta_i) \left(\cos(\phi_k - \phi_i),\sin(\phi_k - \phi_i)\right) = 0.
\end{equation}
A direct calculation with \eqref{alternative} yields
$J \Psi_\theta = \left(\sum_{\ell=1,2,3}m_\ell \sin^2(\theta_\ell)\right)\Psi_\theta$.

Inversely, if S3 is true then
$\Psi_\theta^T J \Psi_\theta=\sum_{\ell=1,2,3}m_\ell \sin^2(\theta_\ell)$,
because $\Psi_\theta^T\Psi_\theta=1$.
From the identity \eqref{identityI}, we obtain $I_{xz}=I_{yz}=0$.
\end{proof}

\begin{remark} 
The above
 equivalences are not affected by the possible degeneracy
of the inertia tensor.
\end{remark}

\begin{remark}{For S3:}
Since the eigenvalue is 
$\lambda=M-\sum_\ell m_\ell \cos^2(\theta_\ell)$,
$M=\sum_{\ell}m_\ell$,
the angles $\cos\theta_k$ are determined uniquely by
\begin{equation}\label{cosTheta}
(\sqrt{m_1}\cos\theta_1,
\sqrt{m_2}\cos\theta_2,
\sqrt{m_3}\cos\theta_3)^T
=\sqrt{M-\lambda}\,\,\Psi_\theta.
\end{equation}
\end{remark}

Now the following result is obvious.
\begin{theorem}[Principal axis]\label{principal} Assuming that 
$\omega \neq 0$, the rotation axis in Definition \ref{def1} is one of the principal axis of the inertia tensor \cite{Routh}.
\end{theorem}
\begin{proof}
For $\omega\ne 0$,
$c_x=c_y=0$ is equivalent to $I_{xz}=I_{yz}=0$.
Then Lemma~\ref{threeEquivalentStatements} gives the proof.
\end{proof}


\section{Eulerian relative equilibrium}\label{sec:Euler}
Most of the results described in this section can be generalized to the  
collinear $n$--body problem. In order to clarify the proofs we will restrict our analysis to the case $n=3$. 

\subsection{Geometry for the Eulerian relative equilibrium ($ERE$)}
We start with the following proposition.

\begin{proposition}
Taking the rotation axis as the $z$--axis, there are just two kinds of collinear $RE.$ All bodies are on the equator or they are on a rotating meridian.
\end{proposition}

\begin{proof}
Take the coordinate system $\xi\eta\zeta$ to avoid a possible confusion with the $xyz$ system in the statement of the Proposition.
Let the masses are on the $\eta=0$ plane.
Then
$I_{\xi\eta}=I_{\eta\zeta}=0$
and
the inertia tensor has the form
\begin{equation}
I=
\left(\begin{array}{ccc}
I_{\xi\xi} & 0 & I_{\xi\zeta} \\
0 & M & 0 \\
 I_{\zeta\xi} & 0 & I_{\zeta\zeta}\end{array}\right),\,\,
 \mbox{ where } \,\,
 M=\sum_\ell m_\ell.
\end{equation}

Then the obvious eigenvector is $e_\eta=(0,1,0)$ which
belongs to the eigenvalue $\lambda=M$.
Taking this eigenvector to indicate the  $z$-axis,  all bodies are on the equator.

The other eigenvectors are on the plane $\eta=0$.
Taking the rotation axis in this plane,
all bodies are on a rotating meridian.
\end{proof}

The above result was first proved just for the cotangent potential in \cite{zhu2}.

For the case when the bodies are on the equator, $\theta_k = \pi/2$ for all $k$, 
and then the relations between the shape variables and the configuration variables are trivial, actually only the shape variables $\sigma_{ij}$ have sense. For this reason our method does not contribute anything new, we can find the $RE$ by
using elementary trigonometric elements see for instance \cite{Diacu1, zhu2}. 
We omit these kind of $RE$ in this paper (see for instance \cite{M-S, EPC2} where for the cotangent potential, even the stability of these $RE$ are studied).

For the case when the bodies are on a rotating meridian,
it is convenient to enlarge the range of $\theta_k$
to $-\pi\le \theta_k\le \pi$ with $\phi_k = 0$. 
The condition $I_{xz}=I_{yz}=0$ is reduced to
$\sum_\ell m_\ell\sin(2\theta_\ell)=0$
which is equivalent to diagonalize the two by two matrix
\begin{equation}
I_2=
\left(\begin{array}{ccc}
\sum_\ell m_\ell \cos^2(\theta_\ell) & -\sum_\ell m_\ell \sin\theta_\ell\cos\theta_\ell \\
-\sum_\ell m_\ell \sin\theta_\ell\cos\theta_\ell  & \sum_\ell m_\ell \sin^2(\theta_\ell)
 \end{array}\right).
\end{equation}

The characteristic polynomial for $I_2$ is
\begin{equation*}
p(\lambda)
=\det(\lambda-I_2)
=\lambda^2-M\lambda
+\sum\nolimits_{i<j}m_im_j\sin^2(\theta_{ij}),
\end{equation*}
where
\begin{equation}
\theta_{ij}=\theta_i-\theta_j.
\end{equation}

The discriminant $D$ for $p(\lambda)=0$ is
\begin{equation}
D=\sum\nolimits_\ell m_\ell^2
	+2\sum\nolimits_{i<j}m_im_j \cos(2\theta_{ij}).
\end{equation}

If $D\ne 0$, the matrix $I_2$ has two distinct eigenvectors.
Therefore, by  Lemma
\ref{threeEquivalentStatements},
the $z$-axis is determined to be one of the two eigenvectors.
So, the angle $\theta_k$ can be determined.
Using the
equation \eqref{fundamentalrelation}, 
we obtain
$\cos \sigma_{ij}=\cos \theta_{ij}$, that is, $\theta_{ij}$ are the shape variables in this case. We obtain the following lemma.
 
\begin{lemma}\label{propTrans} For a collinear $RE$ on a rotating meridian, if $D \neq 0$, then the formulae between the configuration variables $\theta_i$ and the shape variables 
$\theta_{ij}=\theta_i-\theta_j$ are given by 
\begin{equation}\label{collinear-condition}
\left( \cos (2\theta_i),\sin (2\theta_i) \right) = sA^{-1} \sum_{j=1,2,3} m_j \left( \cos (2\theta_{ij}),\sin (2\theta_{ij}) \right), 
\text{ for } i=1,2,3
\end{equation}
where $s = \pm 1$ and $A=\sqrt{D}$.
\end{lemma}
\begin{proof}
If $A\ne 0$, the equation 
$0=\sum_k m_k\sin(2\theta_k)
=\sum_k m_k \sin(2(\theta_1+\theta_{k1}))$
has two solutions
\begin{equation}\label{translationFormula}
\begin{split}
\cos(2\theta_1)&=s A^{-1}\left(
	m_1+m_2\cos\Big(2(\theta_1-\theta_2)\Big)
				+m_3\cos\Big(2(\theta_1-\theta_3)\Big)
	\right),\\
\sin(2\theta_1)&=s A^{-1}\left(
	m_2\sin\Big(2(\theta_1-\theta_2)\Big)
				+m_3\sin\Big(2(\theta_1-\theta_3)\Big)
	\right),
\end{split}
\end{equation}
where $s=\pm1$. The other angles $\theta_k$ are determined by 
$\theta_k=\theta_1+\theta_{k1}$.
\end{proof}

\begin{remark}
The existence of two branches given for $s= \pm 1$ corresponds to the existence of  two  principal axes for $I_2$. 
We will show that the equations of motion determine the sign of $s$.
\end{remark}

\begin{lemma}\label{restri}
For a collinear $RE$ on a rotating meridian,
the shapes that give $D=0$ are restricted to satisfy
\begin{equation}\label{restriction}
m_1+m_2\cos(2\theta_{12})+m_3\cos(2\theta_{13})=0,\,
m_2\sin(2\theta_{12})+m_3\sin(2\theta_{13})=0.
\end{equation}
Or, equivalently
$\sum_{\ell}m_\ell\cos(2\theta_\ell)=\sum_{\ell}m_\ell\sin(2\theta_\ell)=0$.
Therefore, the masses must satisfy the triangle inequalities
$m_k \le m_i+m_j$
for $(i,j,k)=(1,2,3)$, $(2,3,1)$, $(3,1,2)$.
\end{lemma}
\begin{proof}
Since
$D=\Big(m_1+m_2\cos(2\theta_{12})
				+m_3\cos(2\theta_{13})\Big)^2+
	\Big(m_2\sin(2\theta_{12})
				+m_3\sin(2\theta_{13})\Big)^2$,
\eqref{restriction} is obvious.
Then, the other parts follows.
\end{proof}
From now on, the above cyclic expression
will be expressed by
$(i,j,k) \in cr(1,2,3)$.

\subsection{Equations of motion for $ERE$ on a rotating meridian}
Since in this case $\sin(\phi_i-\phi_j)=0$ for all pair $(i,j)$, we do $\phi_k=\omega t$, then the equations of motion for $n=3$ on a rotating meridian are given by
\begin{equation}\label{eqThetaForMeridian0}
\begin{split}
\frac{\omega^2}{2}m_k\sin(2\theta_k)
&= m_k \sum_{j \ne k} m_j\sin(\theta_k-\theta_j)U'( \cos (\theta_k - \theta_j)).\\
\end{split}
\end{equation}

\begin{proposition}\label{shapetoERE}
If $A\ne 0$, 
the equations of motion for $ERE$ are equivalent to the following equations,
\begin{equation}
\label{eqThetaForMeridian}
\begin{split}
&m_1m_2\left(s\frac{\omega^2}{2A}\sin\Big(2(\theta_1-\theta_2)\Big)
- \sin(\theta_1-\theta_2)U'(\cos (\theta_1 - \theta_2))
\right)\\
=&m_2m_3\left(s\frac{\omega^2}{2A}\sin\Big(2(\theta_2-\theta_3)\Big)
- \sin(\theta_2-\theta_3)U'(\cos (\theta_2 - \theta_3))
\right)\\
=&m_3m_1\left(s\frac{\omega^2}{2A}\sin\Big(2(\theta_3-\theta_1)\Big)
- \sin(\theta_3-\theta_1)U'(\cos (\theta_k - \theta_j))
\right).
\end{split}
\end{equation}
\end{proposition}
\begin{proof} 
Using the  formulae \eqref{collinear-condition} between $\theta_k$ and $\theta_i - \theta_j$ given in Lemma \ref{propTrans} when $A\neq 0$ we obtain the result.
\end{proof}

Now we define the useful expressions 
\begin{equation}
\label{defOfFandG}
F_{ij} = m_im_j\sin(\theta_i-\theta_j)U'(\cos \theta_{ij}),\qquad 
G_{ij} = m_im_j\sin (2(\theta_i-\theta_j)).
\end{equation}
Then, the equations of motion 
\eqref{eqThetaForMeridian} can be written in a compact form as
\begin{equation}\label{eqcompact}
s\frac{\omega^2}{2A}(G_{ij}-G_{jk}) = F_{ij} - F_{jk} \quad
\mbox{ for }(i,j,k)\in cr(1,2,3).
\end{equation}

\begin{theorem}
[Condition for a shape]
\label{propConditionForShape}
If $A\ne0$,
\begin{equation}
det
=\left|\begin{array}{cc}
G_{12}-G_{23} & G_{31}-G_{12} \\
F_{12}-F_{23} & F_{31}-F_{12}
\end{array}\right|
=0
\end{equation}
is a necessary and sufficient condition for a shape to satisfy the equations of motion \eqref{eqThetaForMeridian}, and then to generate an $ERE$.
\end{theorem}

\begin{proof}
If equations \eqref{eqcompact} are satisfied, then
$det=0$ is obvious.

Inversely, if $det=0$, then there are two cases.
The first case is when all elements of the matrix are zero.
For this case, the equations of motion 
\eqref{eqcompact}  are trivially satisfied and $\omega$ is undetermined.

The second case 
is when at least one of the elements of the matrix 
is not zero.
For example, let  $G_{12}-G_{23}\ne 0$.
We define
$s\omega^2$ by
$s\omega^2/(2A)=(F_{12}-F_{23})/(G_{12}-G_{23})$.

Then,
$det=0$ yields $s\omega^2/(2A)(G_{31}-G_{12})=F_{31}-F_{12}$.
Thus the equations  \eqref{eqcompact}
are satisfied.
Similarly, all other cases satisfy the equations
\eqref{eqcompact}.
\end{proof}

\begin{remark}
As shown in the previous example,
the sign of $(F_{12}-F_{23})/(G_{12}-G_{23})$ determines $s$.
Namely, the equations of motion determine $s$.
If this value is zero, then $\omega=0$ and the $ERE$ is a fixed point.
\end{remark}

Note that the above equations only depend of the masses and the arc length among the bodies, that is, 
they are free from the choice of the rotation axis
and are the conditions for the shape.

Our method consists in obtain the conditions for a shape that satisfies the equations of motion and then, applying the  formulas \eqref{translationFormula} obtain the corresponding relative equilibria. 

We finish this subsection by setting
the following two corollaries.

\begin{cor}[Equilateral $RE$ on a rotating meridian]\label{equilateralRigidRotator}
For any three masses, the equilateral triangle
$\theta_1-\theta_2=\theta_2-\theta_3=\theta_3-\theta_1
=2\pi/3$ generates a $RE$ on a rotating meridian.
\end{cor}

\begin{proof}
For this shape, $A^2 = \sum_{i<j} (m_i-m_j)^2/2$.
If $m_1=m_2=m_3$ then $A=0$, 
the equations of motion \eqref{eqThetaForMeridian0} are satisfied with $\omega = 0$, the $RE$ is a fixed point. 
For the others cases $A  \neq 0$.  
Then the expressions in the three parentheses of equation \eqref{eqThetaForMeridian} are zero by taking 
$s\omega^2/(2A) = - U'(-1/2)$. 
An alternative proof for the non-equal masses case is obtained by using \eqref{collinear-condition}. We get $\sin (2\theta_k) = -s\sqrt{3}(m_i-m_j)/(2A)$, and we can verify easily that they satisfy 
equation \eqref{eqThetaForMeridian0}.
\end{proof}

\begin{cor}
If a shape on a rotating meridian generates a $RE$ for an
attractive potential $U$, then
the same shape generates a $RE$ for the repulsive potential $-U$.
The corresponding two configurations
differ only by an overall angle of $90$ degrees,
if not, it is a fixed point.
\end{cor}
\begin{proof}
If we change the potential $U$ (attractive) to $-U$ (repulsive),
the $F_{ij}-F_{jk}$ change their  sign,
and the equalities are satisfied by simply changing 
 $s \to -s$ which produces  $\theta_k \to \theta_k+\pi/2$.
\end{proof}


\section{Lagrangian relative equilibrium}\label{sec:Lagrange}

\subsection{Geometry for the Lagrangian relative equilibrium}
For the $n$--body problem on $\mathbb{S}^2$, using the Theorem \ref{principal}, we can determine the candidate for the rotation axis for given masses $m_k$ and shape $\sigma_{ij}$ by computing the principal axes.

In this section we will see that, 
in the three body problem, the equations of motion
determine the special principal axis for the rotation axis among 
all possible candidates.

\subsection{Equations of motion for $LRE$}
 We start by giving the equations of motion for 
$LRE$. From the Lagrangian of the system \eqref{theLagrangian}, 
since $p_{\phi_k} = \omega m_k \sin^2 \theta_k$ is constant,
the equations of $\phi_k$ for a $LRE$ are reduced to 
$\partial V / \partial \phi_k =0$, 
they are given by  
\begin{equation}\label{eqforphi}
\begin{split}
&m_1 m_2 U'(\cos \sigma_{12})\sin\theta_1 \sin\theta_2 \sin(\phi_1-\phi_2)\\
=&m_2 m_3 U'(\cos \sigma_{23})\sin\theta_2 \sin\theta_3 \sin(\phi_2-\phi_3)\\
=&m_3 m_1 U'(\cos \sigma_{31})\sin\theta_3 \sin\theta_1 \sin(\phi_3-\phi_1).
\end{split}
\end{equation}

Analogously, since $\ddot{\theta}_k = 0$, from the Lagrangian we obtain
the equations of motion of $\theta_k$ for a $LRE$:
\begin{equation}\label{eqOfMotForTheta}
\begin{split}
&\omega^2 m_i \sin\theta_i\cos\theta_i\\
&=\sum_{j\ne i}m_i m_jU'(\cos\sigma_{ij})
	\Big(
		\sin\theta_i\cos\theta_j-\cos\theta_i\sin\theta_j\cos(\phi_i-\phi_j)
	\Big).
\end{split}
\end{equation} 

Once we have the equations of motion for the $LRE$, the next question is: what are the conditions to have a $LRE$?

\begin{proposition}\label{propLRE}
The conditions to have a $LRE$ are
\begin{equation}
\label{conditionForLagrangian1}
\sin\theta_k\ne 0
\mbox{ and }
\sin(\phi_i-\phi_j)\ne 0
\mbox{ for all }k
\mbox{ and all pair }i,j, 
\end{equation}
\begin{equation}
\label{conditionForLagrangian2}
U'(\cos \sigma_{12})\cos\theta_3
=U'(\cos \sigma_{23})\cos\theta_1
=U'(\cos \sigma_{31})\cos\theta_2 \ne 0,
\mbox{ and }
\end{equation}
\begin{equation}\label{conditionForLagrangian3}
\cos\theta_k\ne 0
\mbox{ for all }k.
\end{equation}
\end{proposition}

\begin{proof}
From \eqref{eqforphi}, 
since we are assuming 
$U'(\cos \sigma_{ij})$ has definite sign for all $i,j$, then if
$\sin\theta_i\sin\theta_j\sin(\phi_i-\phi_j)=0$ for one pair $(i,j)$,
then the same happen for the others $(i,j)$
which implies that the three bodies must be on the same meridian.
Therefore, by the definition of $LRE$, we obtain \eqref{conditionForLagrangian1}.
Now, by the definition of spherical coordinates $\sin\theta_k>0$,
and using again that $U'(\cos\sigma_{ij})$
has definite sign, we obtain that
all $\sin(\phi_i-\phi_j)$ must have the same sign.
If one configuration with
\begin{equation}\label{signOfDeltaPhi}
\sin(\phi_i-\phi_j)
< 0 \quad 
\mbox{ for } \quad (i,j) \in (1,2), (2,3), (3,1),
\end{equation}
satisfies \eqref{eqforphi} and \eqref{eqOfMotForTheta},
then the configuration that has opposite orientation, 
that is $(\phi_i-\phi_j)\to -(\phi_i-\phi_j)$ and 
$\sin(\phi_i-\phi_j)>0$
also satisfies \eqref{eqforphi} and \eqref{eqOfMotForTheta}.
For concreteness, we assume \eqref{signOfDeltaPhi} in the following.

Multiplying by $\cos(\theta_3)$ both sides of
the first and the second expressions in equation (\ref{eqforphi}),
we obtain
\begin{equation*}
\begin{split}
&m_1m_2U'(\cos \sigma_{12})\sin\theta_1\sin\theta_2\sin(\phi_1-\phi_2)\cos\theta_3\\
&=m_2U'(\cos \sigma_{23})\sin\theta_2
(m_3\sin\theta_3\cos\theta_3\sin(\phi_2-\phi_3))\\
&=m_2U'(\cos \sigma_{23})\sin\theta_2
	(m_1\sin\theta_1\cos\theta_1\sin(\phi_1-\phi_2)).
\end{split}
\end{equation*}
To get the last line, we have used (\ref{alternative}) for $k=2$.
Dividing the first and the last line by
$m_1m_2\sin\theta_1\sin\theta_2\sin(\phi_1-\phi_2)\ne 0$,
we obtain $U'(\cos \sigma_{12})\cos\theta_3=U'(\cos \sigma_{23})\cos\theta_1$.
Similarly, we  obtain, the last part of equation  \eqref{conditionForLagrangian2}.

Again, if one $\cos\theta_k=0$, then all $\cos\theta_k$ must be zero since $U'(\cos \sigma_{ij})\ne 0$ and then 
the configuration is a collinear configuration on the Equator.
Therefore, for the $LRE$ we must have
$\cos\theta_k\ne 0$.
Then by using similar arguments as above, the common value in \eqref{conditionForLagrangian2} is not zero. This finishes the proof of Proposition \ref{propLRE}.
\end{proof}

\begin{cor}
\label{cosMustHaveTheSameSign}
The three bodies for $LRE$ are all on the northern hemisphere
or all on the southern hemisphere.
\end{cor}

\begin{proof}
Since we are assuming 
that $U'(\cos\sigma_{ij})$
has the same sign for all $(i,j)$, we obtain that
the three quantities $\cos\theta_k$ must have the same sign.
\end{proof}

Again, if a configuration with
\begin{equation}\label{positivecosine}
\cos \theta_k > 0 \quad \text{for all} \quad k=1,2,3,
\end{equation}
satisfies the equations of motion \eqref{eqforphi} and \eqref{eqOfMotForTheta}, then
the configuration with $\theta_k \to \pi-\theta_k$ that has $\cos\theta_k<0$
also satisfies the same equations.
For concreteness, we assume \eqref{positivecosine} in the following.

\begin{remark}
The choice $\cos\theta_k>0$ for all $k=1,2,3$ joint with the fact that
$\sin\phi_{ij}$ has the same sign for 
$(i,j) \in (1,2), (2,3),(3,1)$
means that the north pole is inside of the smaller region 
surrounded by the minor 
geodesics connecting the three bodies.
See Figures \ref{figisoscelesSigma12EqPiDiv3and2PiDiv3} and 
\ref{figIsoscelesSigma12EqPiDiv6}.
\end{remark}

\begin{remark}
Corresponding to the choice of the sign for 
$\cos\theta_k$ and $\sin(\phi_i-\phi_j)$,
there are four $LRE$ with the same 
arc angles $\sigma_{ij}$.
\end{remark}

Before proceeding further, we explain why the solution in 
(\ref{conditionForLagrangian2})
deserves the name ``Lagrangian'', of course for this, we have to consider the size of the sphere $\mathbb{S}^2$ measured for its radius $R$.
Consider the Euclidean limit near the North Pole,
$R\to \infty$ with $r_k=R\theta_k$ finite.
Then $\cos\theta_k=1+O(R^{-2})$.
Therefore for this limit,
$U'(\cos \sigma_{12}) \simeq U'(\cos \sigma_{12}) \simeq U'(\cos \sigma_{12})$,
namely $\sigma_{12}\simeq \sigma_{23}\simeq \sigma_{31}$,
that is the equilateral triangle neglecting $O(R^{-2})$ terms.
So, it deserves the name ``Lagrangian''.

Now we give an equivalent system of equations of motion for $LRE$.

\begin{proposition} The equations of motion for $LRE$ are equivalent to
\begin{equation}
\label{equivalentlagrange}
\frac{U'(\cos \sigma_{12})}{\cos\theta_1 \cos\theta_2} =
\frac{U'(\cos \sigma_{23})}{\cos\theta_2 \cos\theta_3} =
\frac{U'(\cos \sigma_{31})}{\cos\theta_3 \cos\theta_1} =
\frac{\omega^2}{\sum_k m_k \cos^2 \theta_k}.
\end{equation}
\end{proposition}

\begin{proof}
By the equation (\ref{conditionForLagrangian2}),
let $U'(\cos \sigma_{ij}) = \gamma \cos\theta_i\cos\theta_j$
with common function $\gamma$.
Substituting this into the equations of motion \eqref{eqOfMotForTheta},
we obtain  the value $\gamma = \omega^2/(\sum_k m_k \cos^2 \theta_k)$.

The inverse is immediate.
\end{proof}

We have seen in Theorem \ref{principal}
that the rotation axis for the $LRE$ is one of the three principal axes of the matrix $J$, the next result 
tells us whom is exactly the rotation axis.

\begin{proposition}\label{prop-eigen}
The rotation axis for a $LRE$ is the eigenvector $\Psi_L$ of $J$ which is given by
\begin{equation}\label{the-eigenvector}
\Psi_L = \left( \frac{\sqrt{m_1}}{U'(\cos \sigma_{23})}, \frac{\sqrt{m_2}}{U'(\cos \sigma_{31})}, \frac{\sqrt{m_3}}{U'(\cos \sigma_{12})} \right)
/
\left( \sum m_k/(U'(\cos \sigma_{ij}))^2 \right)^{1/2},
\end{equation}
where the sum runs for $(i,j,k)\, \in \, cr(1,2,3).$
\end{proposition}

\begin{proof} 
By  equation 
\eqref{conditionForLagrangian2}, $\cos \theta_k = \delta  /U'(\cos \sigma_{ij})$,
with $\delta \ne 0$.
Thus, the normalized eigenvector 
$\Psi_\theta$
in Lemma \ref{threeEquivalentStatements} is given by equation \eqref{the-eigenvector}.
\end{proof}

\begin{remark}
Even when $J$ is degenerated, the equations of motion determine
the rotation axis.
For example for the case $m_k=m$ and $\sigma_{ij}=\pi/2$,
$J$ is $2m$ times the identity matrix. 
Therefore, any direction is a principal axis of $J$.
But,  Proposition \ref{prop-eigen} states that
the equations of motion  determine $\Psi_L=(1,1,1)/\sqrt{3}$ for the rotation axis.
\end{remark}

\begin{cor}
For a $LRE$ the angular velocity $\omega$ is given by  
\begin{equation}\label{omega}
\omega^2 = U'(\cos \sigma_{12})U'(\cos \sigma_{23})U'(\cos \sigma_{31}) \left( \sum m_k/(U'(\cos \sigma_{ij}))^2 \right)^{1/2}. 
\end{equation}
\end{cor}
\begin{proof}
The equation \eqref{equivalentlagrange} and $\Psi_\theta=\Psi_L$ 
yields this expression.
\end{proof}

Now, we have the elements to give
the condition to obtain a $LRE$ from a shape.

\begin{theorem}\label{ConditonForLRE}
The necessary and sufficient condition for a shape given by 
$\sigma_{ij}$
to form a $LRE$ is that the matrix $J$ has the eigenvector $\Psi_L$ given by \eqref{the-eigenvector}.
\end{theorem}

\begin{proof}
The necessity was already shown in Proposition \ref{prop-eigen}. 

Inversely, if the eigenvector is given by the equation \eqref{the-eigenvector}, 
take this eigenvector as the direction  for the $z$--axis.
Then Lemma \ref{threeEquivalentStatements} (S1 $\Rightarrow$ S3) ensures that
$\Psi_\theta$ is also the eigenvector of $J$. Therefore $\Psi_\theta=\Psi_L$
and the eigenvalue is equal to $\sum_\ell m_\ell \sin^2(\theta_\ell)$
which is identified as $\lambda=\Psi_L^T J \Psi_L$.
Then the quantities $\cos \theta_k$ are uniquely determined by  \eqref{cosTheta} as
\begin{equation}
\cos\theta_k =  \sqrt{(M-\lambda)}/
\left(U'(\cos \sigma_{ij})
\sqrt{\sum_{i'j'k'} m_{k'}/(U'(\cos \sigma_{i'j'}))^2}
\right).
\end{equation}
Now, let $\omega^2$ be equal to the value in \eqref{omega}.
Then, we can verify easily that equations \eqref{equivalentlagrange} are satisfied.
\end{proof}

\begin{remark}
As shown above, if $\sigma_{ij}$ turns out to be able to form $LRE$,
$\cos\theta_k$ and $\omega^2$ are uniquely determined by 
 $\sigma_{ij}$. Then, $\phi_i-\phi_j$ are also determined
by \eqref{fundamentalrelation} and \eqref{signOfDeltaPhi}.
Thus, one $LRE$ is determined.
Of course, there are another three $LRE$ that have the same $\sigma_{ij}$
corresponding to the choice of the sign of $\sin(\phi_i-\phi_j)$ and $\cos\theta_k$.
\end{remark}

Before closing this section, we give two corollaries
and a rather trivial proposition.

\begin{cor}[Equilateral $LRE$ requires equal masses]\label{equilateralequal}
An equilateral triangle is a $LRE$ if and only if the masses are equal and the angles $\theta_k$ are equal
\cite{Diacu-EPC1}.
\end{cor}

\begin{proof}
Let $\sigma=\sigma_{12}=\sigma_{23}=\sigma_{31}$. By equation \eqref{conditionForLagrangian2} we have $\cos \theta_1 = \cos \theta_2 = \cos \theta_3$.
Therefore, the matrix $J$ has the eigenvector 
$(\sqrt{m_1},\sqrt{m_2},\sqrt{m_3})$, that is 
\begin{equation*}
\left(\begin{array}{ccc}
m_2+m_3 & -\sqrt{m_1m_2}\cos\sigma & -\sqrt{m_1m_3}\cos\sigma \\
-\sqrt{m_2m_1}\cos\sigma &m_3+m_1& -\sqrt{m_2m_3}\cos\sigma \\
-\sqrt{m_3m_1}\cos\sigma & -\sqrt{m_3m_2}\cos\sigma & m_1+m_2
\end{array}\right)
\!\!\!
\left(\begin{array}{c}
\sqrt{m_1}\\\sqrt{m_2}\\\sqrt{m_3}
\end{array}\right)
=\lambda
\!\!
\left(\begin{array}{c}
\sqrt{m_1}\\\sqrt{m_2}\\\sqrt{m_3}\end{array}\right).
\end{equation*}

This equation reduces to
\begin{equation*}
\lambda
=(m_1+m_2)(\cos\sigma-1)
=(m_2+m_3)(\cos\sigma-1)
=(m_3+m_1)(\cos\sigma-1),
\end{equation*}
which means that $m_1=m_2=m_3$ because $\cos\sigma\ne 1$. 
\end{proof}

\begin{cor}
There are no
 $LRE$ on $\mathbb{S}^2$ for any repulsive force.
\end{cor}
\begin{proof}
Since $U'(\cos \sigma_{ij}) < 0$ for any repulsive force, it can not satisfy 
equation \eqref{omega}.
\end{proof}

\begin{proposition}
There are no fixed point $LRE$.
\end{proposition}
\begin{proof}
Consider a shape given by $\sigma_{ij}$.
Take one of the eigenvectors of $J$ for the $z$-axis.
Then $I_{xz}=I_{yz}=0$ is satisfied.
From here, equation \eqref{equivalentlagrange} must be satisfied for the shape to form a $LRE$ for one of the eigenvectors.
 However, this equation cannot be satisfied by 
 $\omega=0$ and $U'(\cos\sigma_{ij})\ne 0$.
\end{proof}


\section{Equal masses $3$--body $RE$ on $\mathbb{S}^2$ under the
cotangent potential}\label{sec:curved}
The three body problem on $\mathbb{S}^2$ under the cotangent potential
is the natural extension of the Newtonian three body problem to the sphere \cite{Borisov2, Diacu1}.

The cotangent potential and its derivative are given by \cite{Diacu-EPC1}
\begin{equation*}
U(\cos\sigma_{ij})
=\frac{\cos\sigma_{ij}}{\sqrt{1-\cos^2(\sigma_{ij})}},
\qquad
U'(\cos\sigma_{ij})=\frac{1}{\sin^3(\sigma_{ij})}.
\end{equation*}

Since $U'(\cos\sigma_{ij})>0$, this potential produces the attractive force
among the masses.
In this section we only consider the equal masses case $m_k=m$.

The reason why we decided to include this problem in the paper is by one hand, to show how our method is easy to apply and allow us to find new families of relative equilibria. By the other hand, since this problem has been widely studied in the last years (see for instance \cite{Diacu1, Diacu3, Diacu-EPC1, tibboel} and the references therein), 
we can compare the results that we obtain  with the results obtained by other authors using different approaches. In a forthcoming 
paper we will present several new families of relative equilibria for this problem, for now we only show briefly the
equal masses case.

\subsection{Euler relative equilibria}
Equal masses $ERE$ are completely determined by our method.
We will show that there are scalene and isosceles $ERE$.

Since we extended the range of $\theta_k$ to $-\pi\le \theta_k \le \pi$, $U'$ should be written by $U'(\cos\theta_{ij})=1/|\sin^3\theta_{ij}|$.
We write $a=\theta_{21}$ and $x=\theta_{31}$.
Without loss of generality, we can restrict $0<a<\pi$ and $-\pi<x<\pi$.
The discriminant is $D=A^2=3+2(\cos(2a)+\cos(2x)+\cos(2(x-a)))$.

\subsubsection{The degenerated case}
The degenerated case $A=0$ is realized by 
$\cos(2a)=\cos(2x)=-1/2$ and $\sin(2a)+\sin(2x)=0$.
The solutions are the following.
\begin{enumerate}
\item The equilateral triangle: $a=2\pi/3$ and $x=-2\pi/3$.
\item Isosceles triangles with equal angles $\pi/3$. 
Namely, $(a,x)=(\pi/3,2\pi/3)$, $(\pi/3,-\pi/3)$, and $(2\pi/3,\pi/3)$.
\end{enumerate}

For the first shape, the equations of motion are
satisfied if and only if $\omega=0$.
Namely, this shape describes the fixed pint.

For the second shapes, the equations of motion determine
uniquely $\theta_k$ and $\omega$.
For example, for $(a,x)=(2\pi/3,\pi/3)$,
the equations of motion determine
$\theta_3=0$, $\theta_2=-\theta_1=\pi/3$, and $\omega^2=16/(3\sqrt{3})$.

\subsubsection{Non-degenerated case}
For the cases $A\ne0$, the condition for the shape to form an
$ERE$ is
\begin{equation*}
\begin{split}
det=&\frac{g}{(\sin x |\sin x|)(\sin a |\sin a|)(\sin(x-a) |\sin(x-a)|)}=0, \quad \text{where}\\
g=
&\sin(x)|\sin(x)|\Big(\sin(2x)+\sin(2a)\Big)
	\Big(\sin(x-a)|\sin(x-a)|-\sin a |\sin a|\Big)\\
&-\sin(x-a)|\sin(x-a)|\Big(\sin(2a)-\sin(2(x-a))\Big)
	\Big(\sin a |\sin a|+\sin x |\sin x|\Big).
\end{split}
\end{equation*}

\begin{figure}
   \centering
   \includegraphics[width=8cm]{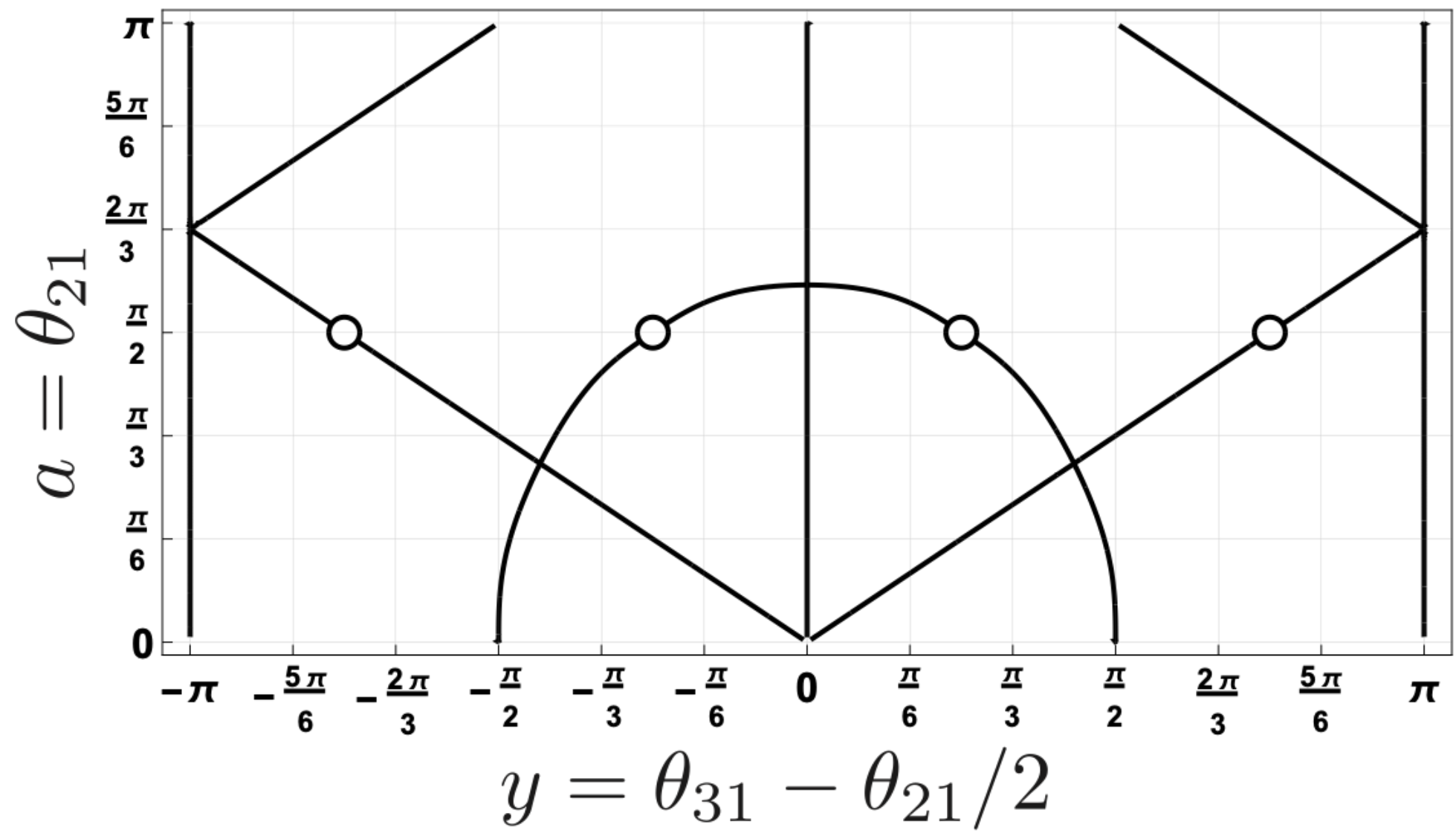}
   \caption{The contour for $det=0$.
   The vertical and horizontal axes represent $a=\theta_{21}$ and
   $y=\theta_{31}-\theta_{21}/2=x-a/2$ respectively.
   Two vertical lines $y=\pm \pi$ represent the same line.
   The curve and the straight lines represent
   scalene and isosceles triangle $ERE$ respectively.
   The four points 
   $(y,a)=(\pm\pi/4,\pi/2), (\pm 3\pi/4,\pi/2)$ are excluded.
   The point $(y,a)=(\pm \pi,2\pi/3)$ represents the equilateral triangle 
   which is the fixed point.}
   \label{figDet}
\end{figure}
\begin{figure}
   \centering
   \includegraphics[width=3.5cm]{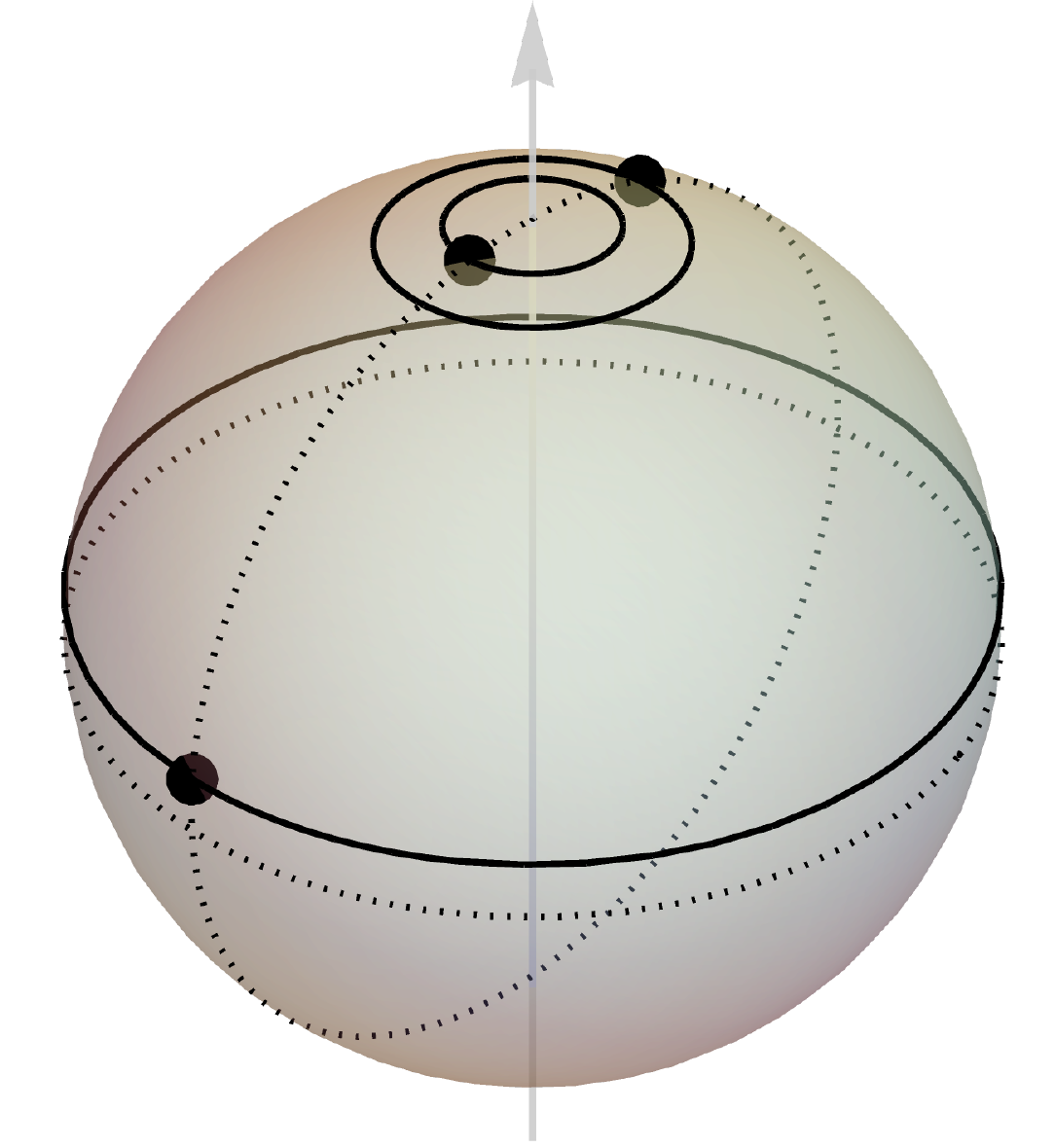}
   \includegraphics[width=3.5cm]{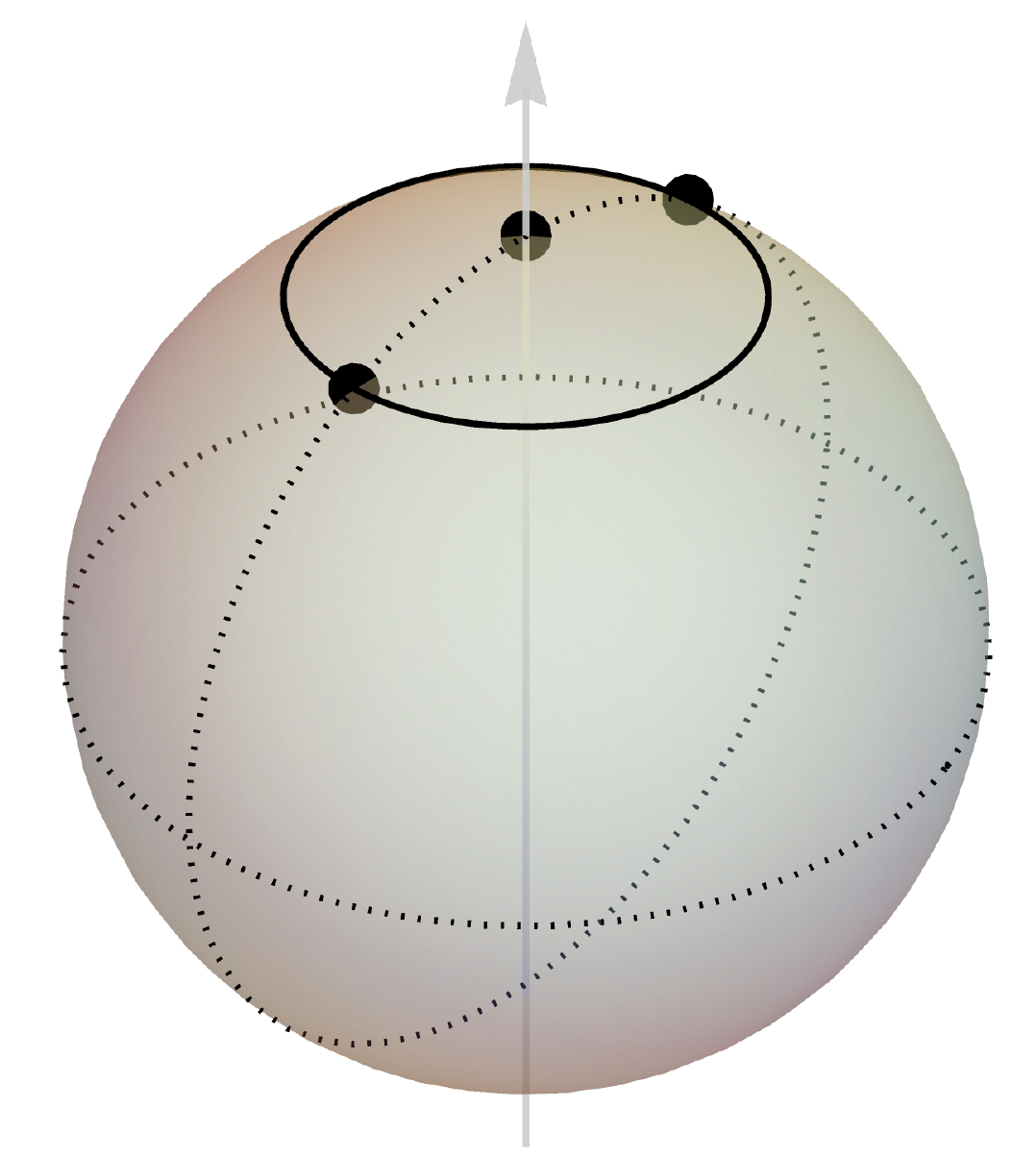}
   \includegraphics[width=3.5cm]{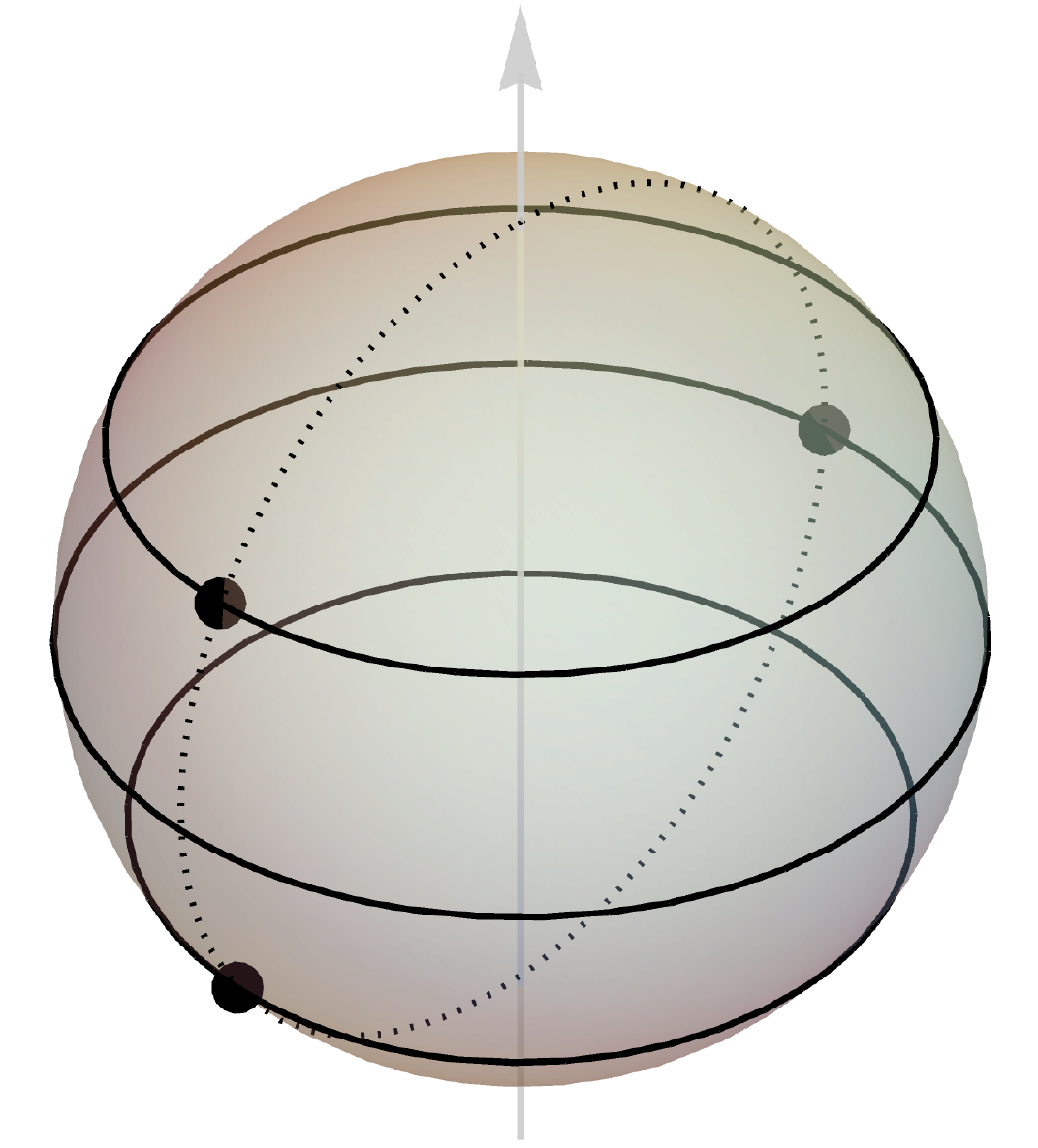}
   \caption{Three typical $ERE$ on a rotating meridean.
   Left: A scalene configuration, where three arc angles between bodies are different.
   Middle: An isosceles configuration with two equal arc angles $\theta$ are smaller than $2\pi/3$ and $\theta\ne \pi/2$, 
   in this case the middle mass must be at one of the poles.
   Right: An isosceles configuration with $2\pi/3<\theta < \pi$, 
   in this case the middle mass must rotates on the equator.}
   \label{Eulerian-scalene}
\end{figure}
The contour $det=0$ is shown in Figure \ref{figDet}.
The curve represents the scalene triangle $ERE$ where all $|\theta_{ij}|$ are different,
and the straight lines represent the isosceles triangle $ERE$.
The equations $det=0$ and $g=0$ are invariant for
the exchange $a \leftrightarrow x$,
or $a\to -a$ and $x\to x-a$.
These invariances are induced by the exchange of masses.

For the scalene triangle, the curve in the region $-a/2<y=x-a/2<a/2$ is represented by
\begin{equation}
\label{solOfCos2y}
\cos(2y)
=\cos(a)+\frac{\sin^2(a)}{\cos(a)}\left(
	\cos(2a)
	+\sqrt{\cos^2(2a)-4\cos(2a)-4}
	\right).
\end{equation}

Let $a_\ell$ be the largest angle of $|\theta_{ij}|$.
The curve for scalene triangle appears when $a_\ell$ satisfies
$\pi/2<a_\ell<a_c=1.8124...$,
where $a_c$ is the solution of equation \eqref{solOfCos2y} for $y=0$.
Namely,
\begin{equation*}
\cos(a_c)
=-1+2^{-1}\left(
				\left(1+9^{-1}\sqrt{78}\right)^{1/3}
				+
				\left(1-9^{-1}\sqrt{78}\right)^{1/3}
				\right).
\end{equation*}
Since $\pi/2<a_\ell$, 
the longest arc length $Ra_\ell$ goes to infinity for $R\to \infty$.
Therefore, the scalene triangle $ERE$ has no  Euclidean limit.

We can solve explicitly the scalene $ERE$ in $-a/2<y<a/2$ for given $\cos a$
using \eqref{solOfCos2y} and usual trigonometric computations.
The solutions for the other regions are given by the symmetry mentioned above.

For the isosceles triangles, we take $m_3$ at the mid point between $m_1$ and $m_2$ with $\theta=\theta_2-\theta_3=\theta_3-\theta_1$, $0<\theta<\pi$.
We can show that
\begin{itemize}
\item[i)] $\theta_3=0 \mod(\pi)$,
$\omega^2=+f(\theta)$ for $\theta\in (0,2\pi/3) 
\setminus
\{\pi/2\}$,
\item[ii)]  $\theta_3$ is arbitrary and $\omega^2=0$ for $\theta=2\pi/3$,
\item[iii)] $\theta_3=\pi/2 \mod(\pi)$,
$\omega^2=-f(\theta)$ for 
$\theta\in (2\pi/3,\pi)$.
\end{itemize}
Where
$f(\theta)=2\left(1/|\sin(2\theta)|^3+1/\big(\sin^2(\theta)\sin(2\theta)\big)\right)$.
See Figure \ref{Eulerian-scalene}.

\subsection{The Lagrange relative equilibria}

Finally we consider  $LRE$ for equal masses case.
The condition for the shape $\sigma_{ij}$ to form $LRE$,
$J\Psi_L=\lambda\Psi_L$ is
\begin{equation}\label{eqForEqualMassLRE}
2-\lambda/m
=\frac{\cos(\sigma_{jk})\sin^3(\sigma_{ki})+\sin^3(\sigma_{jk})\cos(\sigma_{ki})}{\sin^3(\sigma_{ij})},
\end{equation}
for $(i,j,k)\in cr(1,2,3)$.

Our numerical calculations suggest that
there are no scalene triangles solutions for this equation.
However, further investigations are needed to prove this statement.

Let us proceed to the analysis of the isosceles triangles $LRE$.
Let be $\sigma_{23}=\sigma_{31}=\sigma$. 
Then, equation \eqref{eqForEqualMassLRE} is reduced to
\begin{equation}
q(\sigma,\sigma_{12})
=\cos\sigma
\big(2\sin^6(\sigma)-\sin^6(\sigma_{12})\big)
-\sin^3(\sigma)\cos(\sigma_{12})\sin^3(\sigma_{12})
=0.
\end{equation}
The graphical representation of this equation is
shown in Figure \ref{figMspecialCase8equalMassesCaseCont3FigSolution}.
\begin{figure}
   \centering
   \includegraphics[width=7cm]{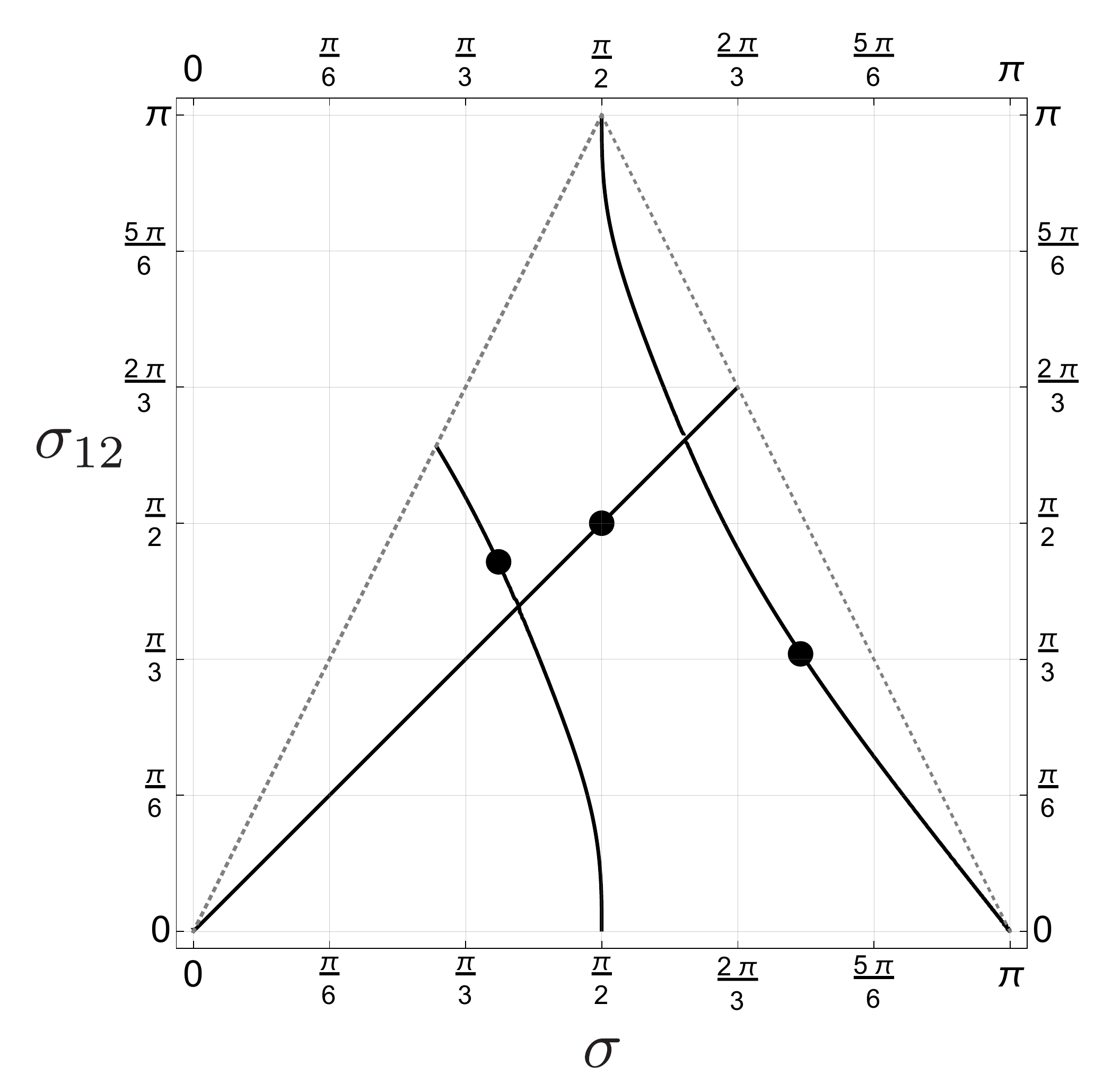} 
   \caption{The solid line represents the shape for the equal masses isosceles 
   $LRE$
   with $\sigma_{12}$ and $\sigma=\sigma_{23}=\sigma_{31}$.
   The straight line represents
   equilateral triangle.
   The region inside the 
  dotted lines is 
   the region to form a triangle, 
   $\sigma_{12}<2\sigma<2\pi-\sigma_{12}$.
   Note that the curve is point symmetric
   around $(\sigma,\sigma_{12})=(\pi/2,\pi/2)$.
   The three  black circles represent
   the right-angled triangles,
   where the angle at the vertex
   $m_3$ is $\pi/2$.
   }
   \label{figMspecialCase8equalMassesCaseCont3FigSolution}
\end{figure}

\begin{figure}
   \centering
\includegraphics[width=3.5cm]{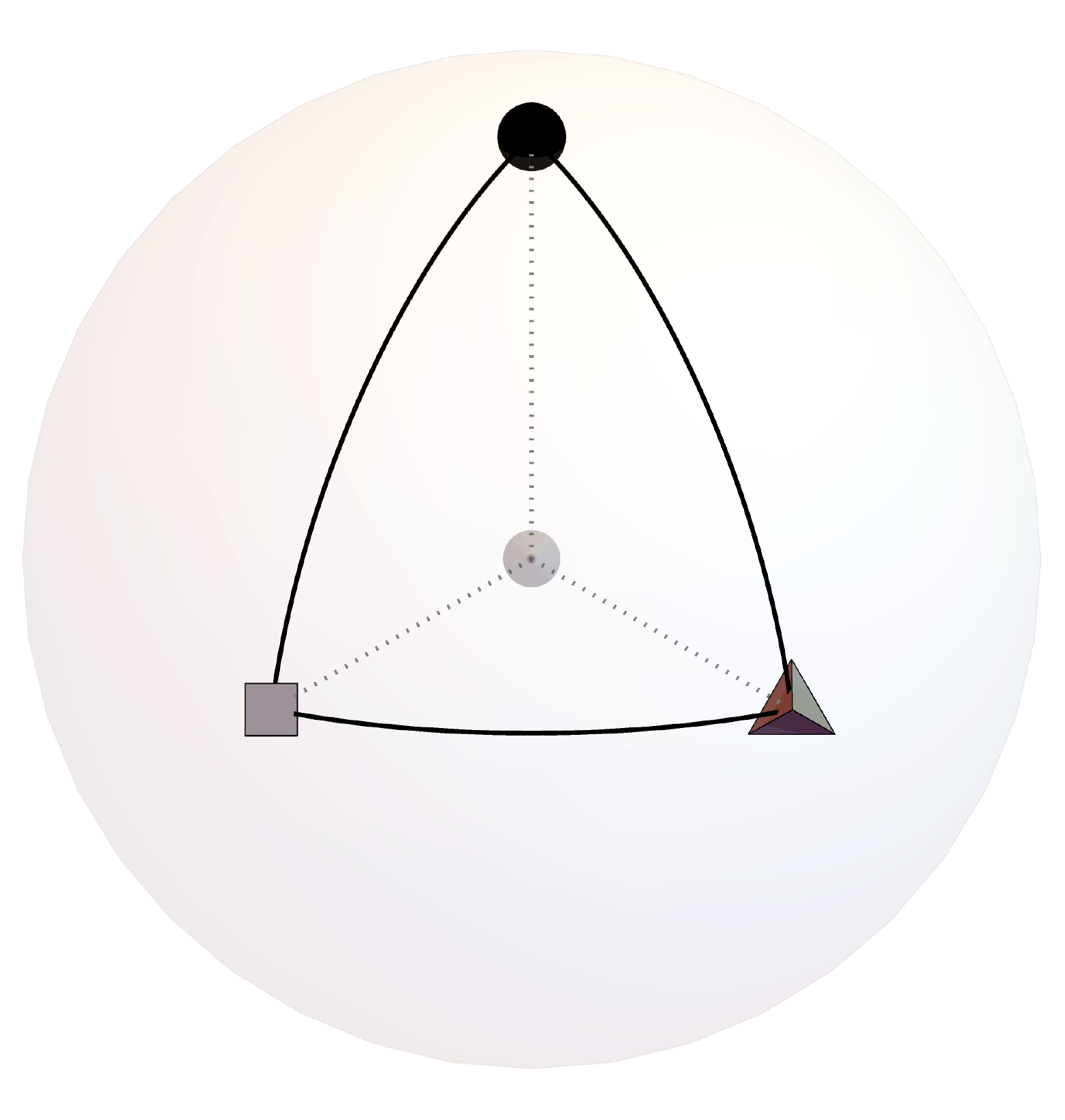}
\includegraphics[width=3.5cm]{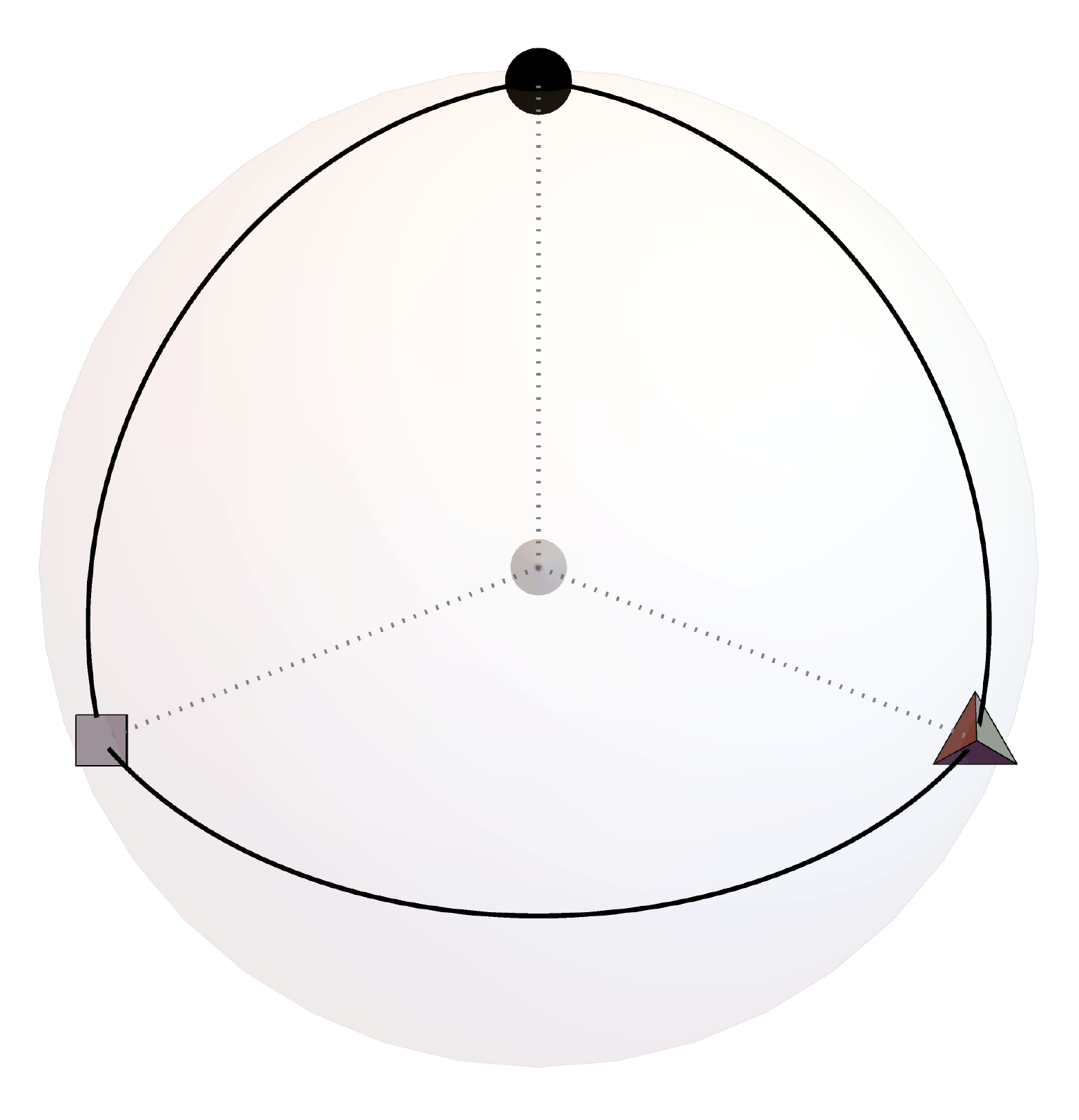}
   \caption{
   An example for  pair of equal mass isosceles $LRE$
   predicted by the symmetry
   $q(\sigma,\sigma_{12})=0 \Leftrightarrow q(\pi-\sigma,\pi-\sigma_{12})=0$
   seen from above the North pole.
   The left is $\sigma_{12}=\pi/3$ and $\sigma_{23}=\sigma_{31}=1.33240...$.
   The right is $\sigma_{12}=2\pi/3$ and 
   $\sigma_{23}=\sigma_{31}=\pi-1.33240...=1.80918...$.
   The grey ball at the centre represents the North pole.
   They have the common angular velocity $\omega^2=3.85072...$.}
   \label{figisoscelesSigma12EqPiDiv3and2PiDiv3}
\end{figure}
It is obvious that if $(\sigma, \sigma_{12})$ is a solution of 
$q(\sigma,\sigma_{12})=0$,
then $(\pi-\sigma, \pi-\sigma_{12})$ is also a solution.
Namely, $q(\sigma,\sigma_{12})=0$ is point symmetric
around $(\sigma,\sigma_{12})=(\pi/2,\pi/2)$.
Since $U'(\cos\sigma_{ij})=1/\sin^3(\sigma_{ij})$,
the angular velocity $\omega^2$ given by \eqref{omega} is invariant by this symmetry.
See Figure \ref{figisoscelesSigma12EqPiDiv3and2PiDiv3}.

In Figure \ref{figIsoscelesSigma12EqPiDiv6} we show three isosceles $LRE$ with $\sigma_{12}=\pi/6$.

 \begin{figure}
    \centering
\includegraphics[width=3cm]{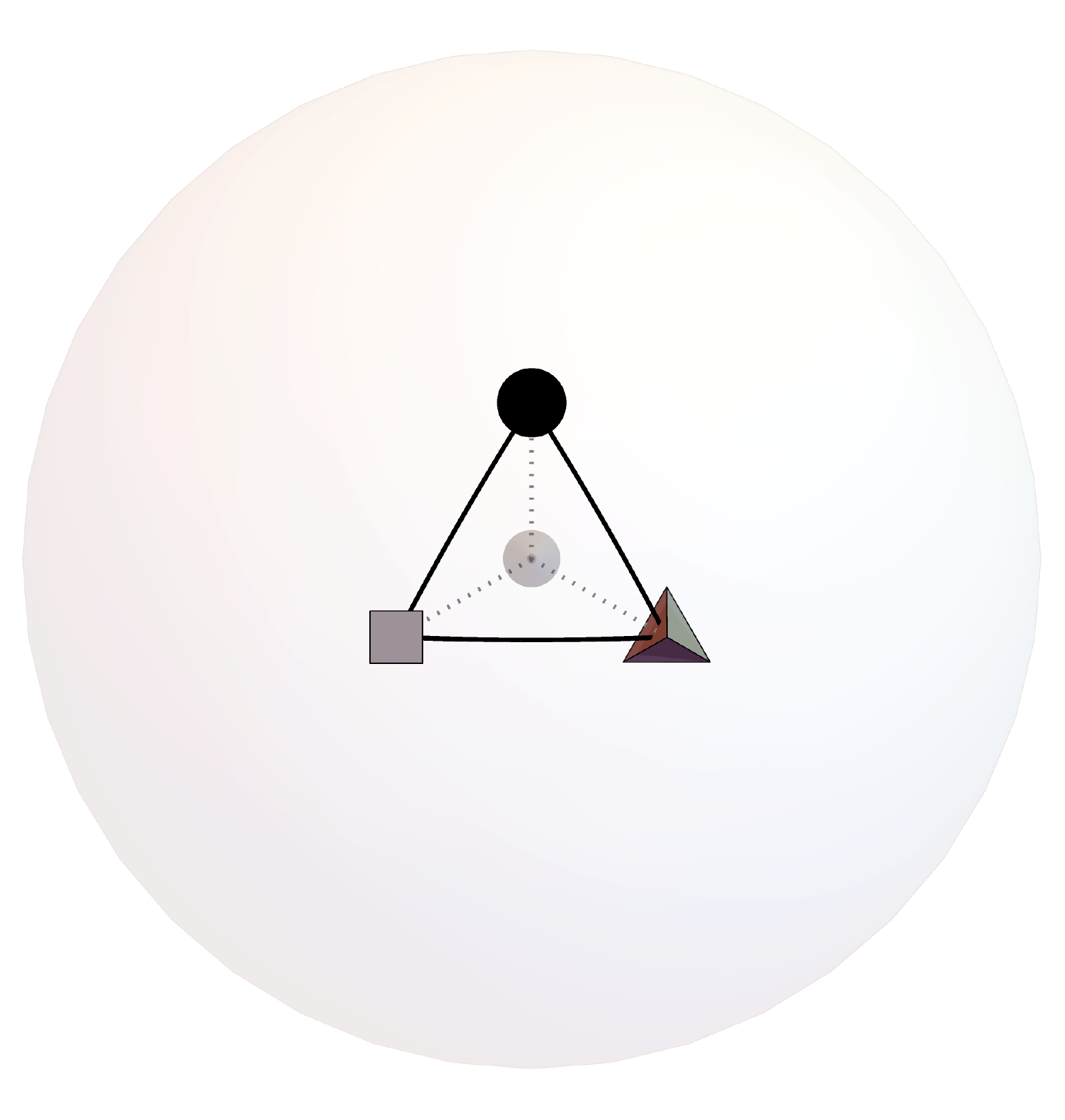}
\includegraphics[width=3cm]{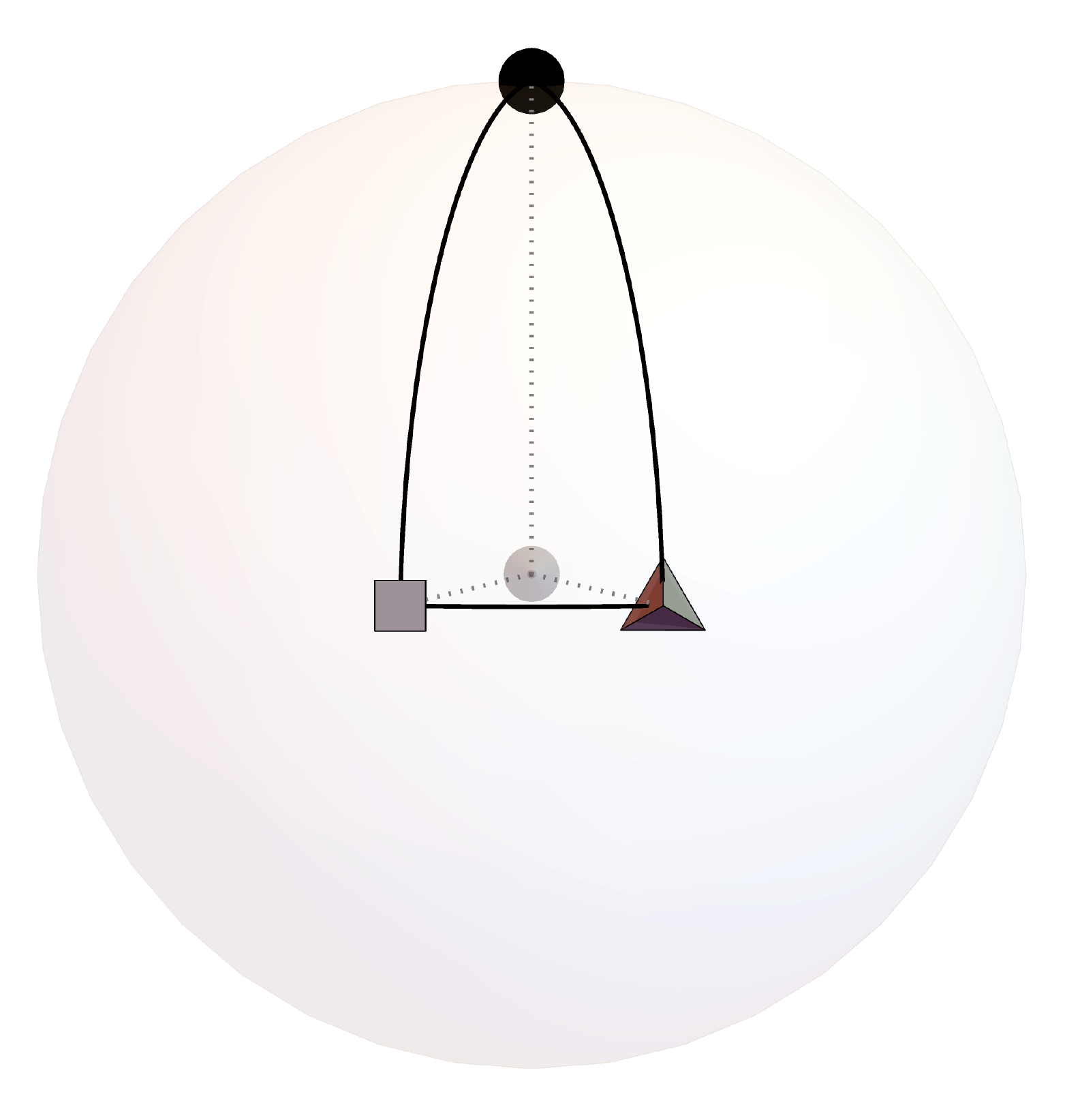}
\includegraphics[width=3cm]{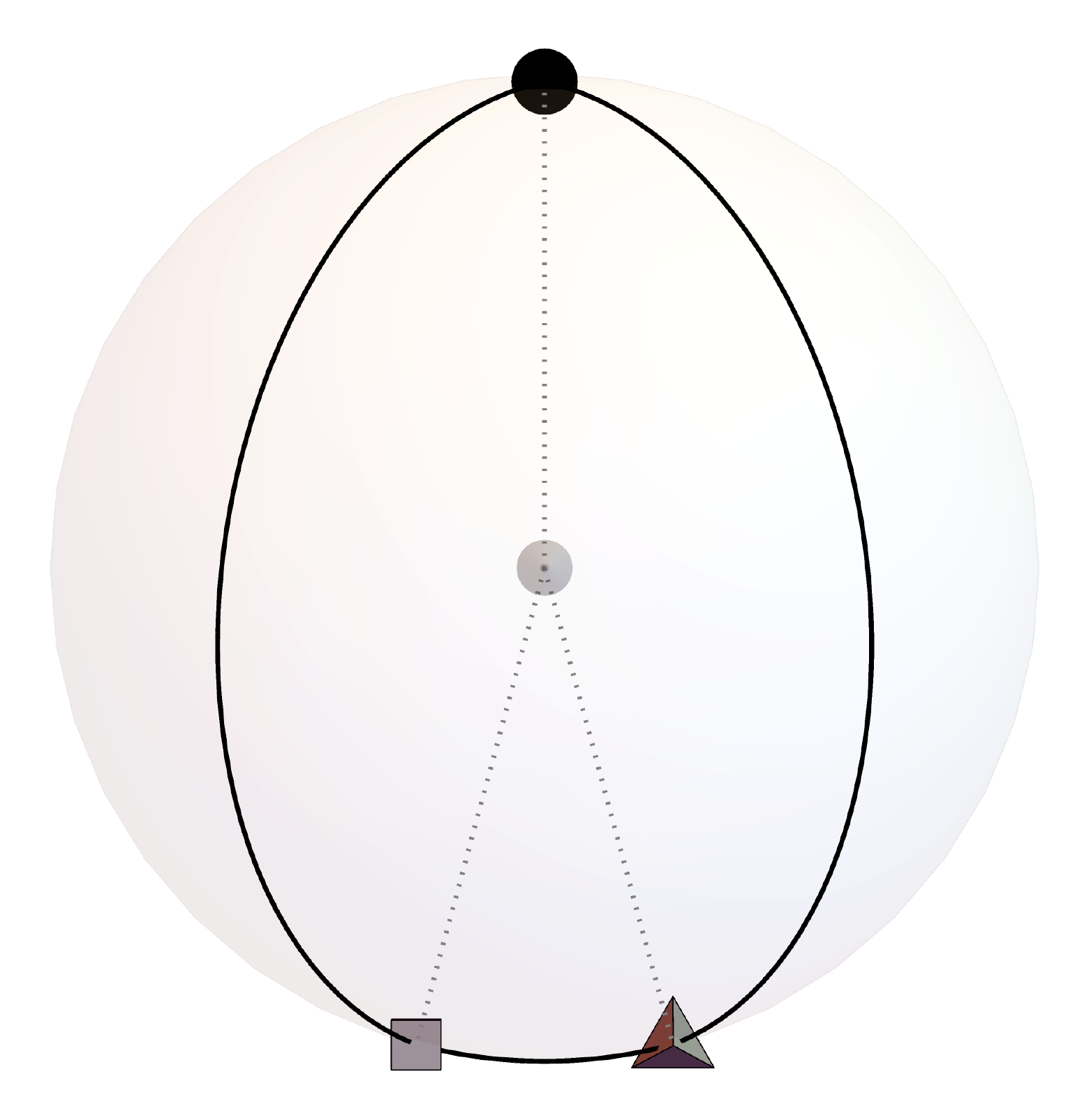} 
    \caption{Three isosceles Lagrangian $RE$ configurations
    with $\sigma_{12}=\pi/6$
    for equal masses,
    seen from above the North pole.
    From left to 
    right
    $\sigma_{23}=\sigma_{31}=
    \pi/6$ (equilateral triangle), $1.51596...\leq \pi/2 $, 
    and $2.73083...\geq 5\pi/6$.}
    \label{figIsoscelesSigma12EqPiDiv6}
 \end{figure}

\section{Conclusions and final remarks}\label{conclusions}
We successfully derive the conditions for general shapes of $n$--bodies on the sphere $\mathbb{S}^2$ to generate relative equilibria. 
Since
there are no translational invariance on $\mathbb{S}^2$, 
the linear momentum (from where we obtain the center of mass) is not more a first integral for the  $n$--body problem on $\mathbb{S}^2$.
The lack of the center of mass
was an obstacle to study relative equilibria on $\mathbb{S}^2$, because a priori we do not know how to choose the rotation axis. 
We solved this problem by showing
that the condition $c_x=c_y=0$ 
determines the rotation axis.

By introducing the inertia tensor $I$, we show that the rotation axis for a relative equilibrium is one of the principal axis of $I$. We divide the analysis of relative equilibria on the sphere into two big classes, collinear and non-collinear. 
In both cases, for $n=3$, we give the necessary and sufficient conditions to obtain a $RE$ for a given shape.

To show how our method works to determine $RE$ on $\mathbb{S}^2$,
we study the equal masses case for the cotangent potential.
The $ERE$ is completely determined, and $LRE$ is almost.
The remaining problem is whether scalene $LRE$ exist or not.

In our method, we first determine the shape, then we
obtain the corresponding configuration.
This procedure is similar to
 the method to obtain 
the $RE$ in the Euclidean plane.
Actually, to obtain $ERE$ in the Euclidean plane, 
we first solve the famous Euler
fifth order equation \cite{Euler,Hestenes,Moeckel}
to determine the mutual distances $r_{ij}$,
then we obtain the corresponding configuration
using the fact that the rotation center is the center of mass.

For the analysis of the $ERE$ on the sphere, 
the condition $det=0$
is a natural extension of  the Euler's fifth order equation.
For the cotangent potential, we have

\begin{equation*}
\begin{split}
&det=\frac{m_1m_2m_3 \,g}
{\sin\theta_{12}\sin\theta_{23}\sin\theta_{31}
|\sin\theta_{12}\sin\theta_{23}\sin\theta_{31}|},\\
&g=\sum m_k \sin\theta_{ij}|\sin\theta_{ij}|
\big(\sin\theta_{ki}|\sin\theta_{ki}|\sin(2\theta_{ki})
-\sin\theta_{jk}|\sin\theta_{jk}|\sin(2\theta_{jk})
\big),
\end{split}
\end{equation*}
where the 
sum runs for $(i,j,k)\in cr(1,2,3)$.
Note that $g=0$ is a fifth order equation
for $\sin\theta$.
We can easily verify that 
 limit $R\to \infty$ with $R\, |\theta_{ij}|=r_{ij}$ fixed,
yields the Euler's fifth order 
equation.

\subsection*{Acknowledgements}
Thanks to our friend Florin Diacu, 
who was the inspiration of this work.
The second author (EPC) has been partially supported 
by Asociaci\'on Mexicana de Cultura A.C. and Conacyt-M\'exico Project A1S10112.

\end{document}